\documentclass[onecolumn,11pt]{article}
\usepackage[top=1.2in, bottom=1.2in, left=1in, right=1in]{geometry}
\usepackage{amsmath,amsfonts,amscd,amssymb}
\usepackage{graphicx}
\usepackage{mathdots}
\usepackage{epstopdf}
\usepackage{psfrag}
\usepackage{overpic}
\usepackage{cancel}
\usepackage{rotating}
\usepackage{url}
\usepackage{caption}
\usepackage{color}
\usepackage{rotating}
\usepackage{multirow}
\usepackage{wrapfig}
\usepackage{mathtools}
\usepackage{subeqnarray}
\usepackage{setspace}
\usepackage{palatino} 
\usepackage[numbers,sort&compress]{natbib}
\usepackage[table]{xcolor}    
\usepackage[numbered,framed]{matlab-prettifier}
\usepackage{booktabs}
\usepackage{authblk}

\newcommand{\Rey}[0]{Re}
\newcommand{\bT}{\boldsymbol{T}}
\usepackage{array}
\newcolumntype{L}{>{\centering\arraybackslash}m{10cm}}

\usepackage[bottom,flushmargin,hang,multiple]{footmisc}
\usepackage{lipsum}
\newcommand\blfootnote[1]{%
  \begingroup
  \renewcommand\thefootnote{}\footnote{#1}%
  \addtocounter{footnote}{-1}%
  \endgroup
}

\definecolor{header1}{cmyk}{0,0,0,1}

\usepackage{amsthm}
\newtheorem{theorem}{Theorem}[section]

\newtheorem{lemma}[theorem]{Lemma}
\usepackage{mleftright}
\DeclareMathOperator*{\argmax}{arg\rm{}max}
\DeclareMathOperator*{\argmin}{arg\rm{}min}

\newcommand{\cM}{\mathcal{M}}
\newcommand{\diag}{\textrm{diag}}
\newcommand{\bL}{\boldsymbol{L}}
\newcommand{\bJ}{\boldsymbol{J}}
\newcommand{\bI}{\boldsymbol{I}}

\newcommand{\bd}{\boldsymbol{d}}
\newcommand{\bc}{\boldsymbol{c}}
\newcommand{\bC}{\boldsymbol{C}}
\newcommand{\bA}{\boldsymbol{A}}

\newcommand{\bX}{\boldsymbol{X}}
\newcommand{\bFc}{\boldsymbol{\mathcal{F}}}
\newcommand{\bYc}{\boldsymbol{\mathcal{Y}}}
\newcommand{\bXc}{\boldsymbol{\mathcal{X}}}
\newcommand{\bRc}{\boldsymbol{\mathcal{R}}}
\newcommand{\bSc}{\boldsymbol{\mathcal{S}}}

\newcommand{\bY}{\boldsymbol{Y}}
\newcommand{\bF}{\boldsymbol{F}}

\newcommand{\bR}{\boldsymbol{R}}

\newcommand{\bQ}{\boldsymbol{Q}}

\newcommand{\bV}{\boldsymbol{V}}

\newcommand{\bS}{\boldsymbol{S}}

\newcommand{\bSigma}{\boldsymbol{\Sigma}}
\newcommand{\bU}{\boldsymbol{U}}
\newcommand{\bb}{\boldsymbol{b}}

\newcommand{\bx}{\boldsymbol{x}}
\newcommand{\bLambda}{\boldsymbol{\Lambda}}
\newcommand{\ba}{\boldsymbol{a}}
\newcommand{\bH}{\boldsymbol{H}}

\newcommand{\E}{\mathbb{E}}

\newcommand{\bh}{\boldsymbol{h}}

\newcommand{\bPsi}{\boldsymbol{\Psi}}

\newcommand{\by}{\boldsymbol{y}}
\newcommand{\bv}{\boldsymbol{v}}
\newcommand{\bu}{\boldsymbol{u}}

\newcommand{\bxi}{\boldsymbol{\xi}}

\newcommand{\bD}{\boldsymbol{D}}

\newcommand{\Var}{\mathrm{Var}}

\newcommand\e{\textrm{e}}

\usepackage{amssymb}
\usepackage{lipsum}
\usepackage{tikz,pgfplots,pstool}
\usetikzlibrary{shapes.misc}
\tikzset{cross/.style={cross out, draw=blue, minimum size=2*(#1-\pgflinewidth), inner sep=0pt, outer sep=0pt},
	cross/.default={1pt}}

\pgfplotsset{compat=1.13}

\usepackage{contour}
\usepackage[normalem]{ulem}

\contourlength{0.8pt}

\newcommand{\myuline}[1]{%
  \uline{\phantom{#1}}%
  \llap{\contour{white}{#1}}%
}

\usepackage{subcaption}

\newtheoremstyle{mylemmastyle}
  {3pt} 
  {3pt} 
  {\itshape} 
  {} 
  {\bfseries} 
  {.} 
  {\newline} 
  {\thmname{#1} \thmnumber{#2}: \thmnote{\normalfont\myuline{#3}}} 

\theoremstyle{mylemmastyle}

\usepackage{algorithm}
\usepackage{algorithmicx}
\usepackage{algpseudocode}
\algrenewcommand{\algorithmiccomment}[1]{\hfill #1}

\usepackage{hyperref}
\usepackage{xcolor}
\hypersetup{
    colorlinks,
    linkcolor={red!70!black},
    citecolor={blue!70!black},
    urlcolor={green!50!black}
}
\renewcommand\d{\textrm{d}}
\definecolor{col0}{RGB}{82,173,2}
\definecolor{col1}{RGB}{212,173,2}
\definecolor{col2}{RGB}{255, 111, 0}

\title{\vspace{-.4in}\textbf{Physics-informed dynamic mode decomposition (piDMD)}\vspace{-.1in}}
\author[1]{ Peter J. Baddoo$^*$}
\author[2]{ Benjamin Herrmann}
\author[3]{ Beverley J. McKeon}
\author[4]{ J. Nathan Kutz}
\author[5]{ Steven L. Brunton}

\affil[1]{\small Department of Mathematics, Massachusetts Institute of Technology,
 Cambridge, MA 02139, USA}
\affil[2]{Department of Mechanical Engineering, University of Chile, Beauchef 851, Santiago, Chile}
\affil[3]{Graduate Aerospace Laboratories, California Institute of Technology, Pasadena CA 91125, USA}
\affil[4]{Department of Applied Mathematics, University of Washington, Seattle, WA 98195, USA}
\affil[5]{Department of Mechanical Engineering, University of Washington, Seattle, WA 98195, USA}

\date{}
\renewcommand{\i}{\textrm{i}}
\begin{document}
\maketitle
\blfootnote{$^*$ Corresponding author (baddoo@mit.edu).}
\date{}
\vspace{-1.25cm}
\begin{abstract}
	In this work, we demonstrate how physical principles -- such as symmetries, invariances, and conservation laws -- can be integrated into the {\em dynamic mode decomposition} (DMD). 
	DMD is a widely-used data analysis technique that extracts low-rank modal structures and dynamics from high-dimensional measurements.
    However, DMD frequently produces models that are sensitive to noise, fail to generalize outside the training data, and violate basic physical laws. 
    Our physics-informed DMD (piDMD) optimization, which may be formulated as a Procrustes problem, restricts the family of admissible models to a matrix manifold that respects the physical structure of the system. 
    We focus on five fundamental physical principles --  
	conservation, self-adjointness, localization, causality, and shift-invariance --
    and derive several closed-form solutions and efficient algorithms for the corresponding piDMD optimizations.
    With fewer degrees of freedom, piDMD models are less prone to overfitting, require less training data, and are often less computationally expensive to build than standard DMD models.  
	We demonstrate piDMD on a range of challenging problems in the physical sciences, including energy-preserving fluid flow, travelling-wave systems, the Schrödinger equation, solute advection-diffusion, a system with causal dynamics, and three-dimensional transitional channel flow.
	In each case, piDMD significantly outperforms standard DMD in metrics such as spectral identification, state prediction, and estimation of optimal forcings and responses.
\end{abstract}

\section{Introduction}
\label{Sec:Intro}
Integrating partial knowledge of physical principles into data-driven techniques is a primary goal of the 
scientific machine learning (ML) community~\cite{Karniadakis2021}.
Physical principles -- such as conservation laws, symmetries, and invariances -- can be incorporated into ML
algorithms in the form of inductive biases, thereby ensuring that the learned models are constrained to the correct physics.
Recent successful examples of ML algorithms that have been modified to respect physical principles include
neural networks~\cite{Behler2007,Greydanus2019,Reichstein2019,Raissi2019,Cranmer2020,Raissi2020,Wang2020,Lu2021,Li2021},
kernel methods~\cite{BaddooLANDO,Klus2021},
deep generative models~\cite{Shah2019}, 
and sparse regression~\cite{Brunton2016,Rudy2017,Loiseau2018,guan2020sparse,zanna2020data}.
These examples demonstrate that incorporating partially-known physical principles into machine learning architectures
can increase the accuracy, robustness, and generalizability of the resulting models, while simultaneously
decreasing the required training data.
In this work, we integrate knowledge of physical principles into one of the most widely-used methods in data-driven dynamical systems research: the dynamic mode decomposition~\cite{Schmid2010,Rowley2009,Tu2014,Kutz2016,Askham2018,Schmid2022}.

The dynamic mode decomposition (DMD) is a data diagnostic technique that 
extracts coherent spatial-temporal patterns from high-dimensional time series data~\cite{Schmid2010,Schmid2022}.
Although DMD originated in the fluid dynamics community~\cite{Schmid2010},
the algorithm has since been applied to a wealth of dynamical systems including in
epidemiology~\cite{proctor2015ih}, 
robotics~\cite{berger2014ieee,abraham2019ieee},
neuroscience~\cite{brunton2016extracting}, 
quantum control~\cite{Goldschmidt2021},
power grids~\cite{susuki2011a},
and plasma physics~\cite{Taylor2018,kaptanoglu2020pop}.
Despite its widespread successes, DMD is highly sensitive to noise~\cite{Bagheri2014pof,Dawson2016,Hemati2017tcfd},
fails to capture travelling wave physics,
and can produce overfit models that do not generalize.
Herein, we demonstrate that integrating physics into the learning framework can help address these challenges.

Suppose that we are studying a dynamical system defined by $\dot{\bx} = \bF(\bx)$ (continuous time) or $\bx_{k+1} = \bF(\bx_k)$ (discrete time) where $\bF: \mathbb{R}^n \rightarrow \mathbb{R}^n$ is unknown. 
DMD identifies the best low-rank linear approximation of $\bF$ given a collection of $m$ pairs of measurements $\{\bx_j, \by_j\}$; 
in other words, DMD seeks a rank $r$ matrix $\bA \in \mathbb{R}^{n\times n}$ such that 
\begin{align}\label{Eq:DMDAMatrix}
    \by_{j} \approx \bA \bx_j
    \end{align}
    for $j=1,\dots, m$.
Arranging the data measurements into $n\times m$ matrices $\bX$ and $\bY$ allows us to phrase the above formally as 
\begin{align}
	\textnormal{\myuline{DMD regression:}}\qquad \qquad
	&\argmin_{\textrm{rank }(\bA) = r} \|\bY - \bA \bX \|_F. \qquad \qquad \qquad\label{Eq:DMD}
\end{align}
After approximately solving \eqref{Eq:DMD}, the DMD process computes the dominant spectral properties of the learned linear operator
\cite{Schmid2010,Rowley2009,Tu2014,Kutz2016,Askham2018,Schmid2022}.
The rank-$r$ constraint in \eqref{Eq:DMD} is motivated by the assumed modal structure of the system but doesn't account for other important physical properties. 
For example, one limitation of DMD is that the solution of \eqref{Eq:DMD} lies within the span of $\bY$, so the learned model rarely generalizes outside the training regime. We address this limitation and others by embedding partial physics knowledge into the learning process.

In this work, we incorporate physical principles into the optimization \eqref{Eq:DMD} by constraining the solution matrix $\bA$ to
lie on a matrix manifold $\cM \subseteq \mathbb{R}^{n\times n}$:
\begin{align}
	\textnormal{\myuline{piDMD regression:}} \qquad  \qquad
	&\argmin_{\bA \in \cM} \|\bY - \bA \bX \|_F .\qquad \qquad\quad \qquad \label{Eq:procrustes}
\end{align}
\begin{figure}[t]
\vspace{-.4cm}
	\centering
	\includegraphics[width=.83\linewidth]{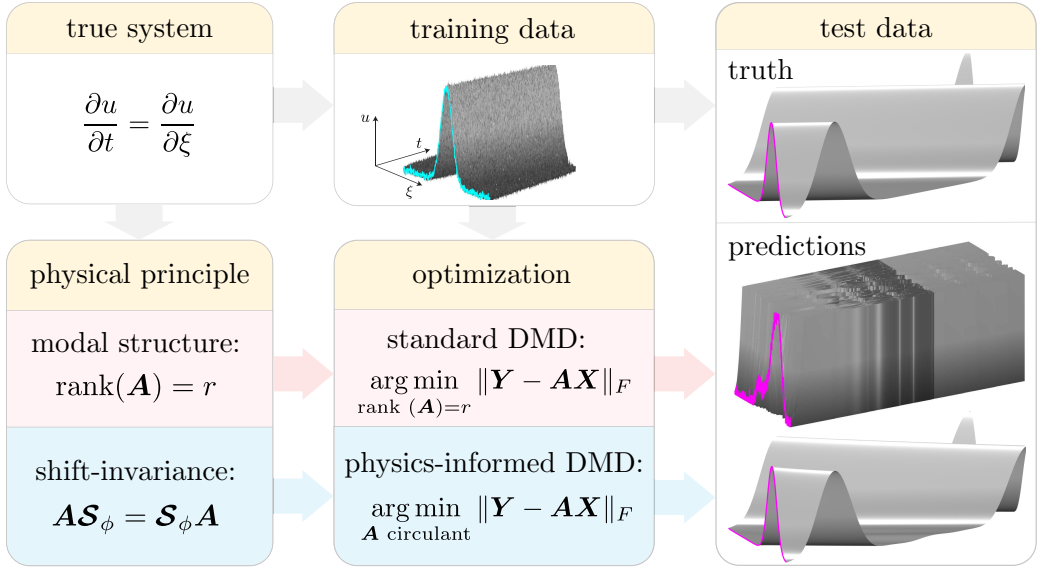}
	\vspace{-.00in}
	\caption{Comparing standard dynamic mode decomposition (specifically, optimized DMD \cite{Askham2018})
	and physics-informed dynamic mode decomposition applied to the 
advection equation.
The data is contaminated with 2\% additive Gaussian noise.
Having trained DMD models, we then perform predictions for different initial conditions.
Standard DMD fails whereas piDMD produces a faithful prediction. In the figure, $\boldsymbol{\mathcal{S}}_\phi$ represents the shift operator (see \S \ref{Sec:shiftInvariant}).
}%
\vspace{-.2cm}
\label{Fig:schematic}
\end{figure}%
\noindent The matrix manifold $\cM$ is dictated by the known physics of the system at hand. 
For example, we may select $\cM$ such that its members satisfy certain symmetries that we known are obeyed by the system at hand.
We call \eqref{Eq:procrustes} `\emph{physics-informed} DMD' (piDMD) as the optimization integrates 
underlying knowledge of the system physics into the learning framework.
Again, the spectral features of the solution of \eqref{Eq:procrustes} can be computed to give insight into the modes that dominate the system and their dynamics.
The low-rank DMD regression \eqref{Eq:DMD} is a special case of piDMD where $\cM$ is the manifold of rank-$r$ matrices.
However, piDMD models are not always low rank (see, for example, \S \ref{Sec:local}) but nevertheless have few degrees of freedom due to the physics-informed constraints.
Constraining the solution matrix improves the generalizability of the learned model,
reduces sensitivity to noise and reduces the demand for large training sets.
More broadly, constraining a linear, dimensionality-reduced model to a manifold $\cM$ has been shown to result in a variety of optimization formulations and techniques which can improve low-rank approximations~\cite{cunningham2015linear}.
In figure \ref{Fig:schematic} we compare the performance of standard \eqref{Eq:DMD} and physics-informed \eqref{Eq:procrustes} DMD
for a simple travelling wave.
We constrain the piDMD model to the shift-invariant symmetry of the underlying system.
As a result, the piDMD model produces an accurate prediction of the future dynamics, whereas the standard DMD model blows up immediately.
This simple example illustrates that embedding basic physics into DMD can substantially improve the algorithm's performance.

The optimization problem \eqref{Eq:procrustes} is well-known in the statistical literature as a \mbox{\emph{Procrustes problem}}~\cite{Hurley1962,Soderkvist1993,Trendafilov2002,Gower2007,Higham1988,TenBerge1993,Arun1992,Gillis2018,Schonemann1966,Pumir2021,Viklands2006,Elden1999,Andersson2006}.
Our recasting of the DMD regression as a Procrustes problem
is a new connection, and is the basis of piDMD.
This perspective enables us to leverage the substantial extant work on Procrustes problems into new application areas.
The literature contains many exact solutions for Procrustes problems,
including notable cases of orthogonal matrices~\cite{Schonemann1966},
and symmetric matrices~\cite{Higham1988}.

A significant challenge in physics-informed ML architectures is to develop implementations that 
incorporate known physics but also scale well to higher dimensions~\cite{Karniadakis2021}.
A major contribution of this work is several exact solutions, in terms of standard linear algebra operations, of 
Procrustes problems of physical relevance.
Where possible, we have made use of matrix factorizations and low-rank representations to alleviate the computational burden of implementation;
it is rarely necessary to form $\bA$ explicitly when computing its spectral properties or other diagnostics. 
For complicated matrix manifolds, an exact solution of the Procrustes
problem \eqref{Eq:procrustes} may be intractable and we must resort to algorithmic approaches.
Fortunately, optimization on matrix manifolds is a mature field~\cite{Absil2009},
and many algorithms are implemented in the open source
software package `Manopt' \cite[\url{www.manopt.org}]{Boumal2014}.

The remainder of the paper is arranged as follows.
In section \ref{Sec:background} we provide background information on DMD and Procrustes problems.
Then, in section \ref{Sec:piDMD}, we describe the broad framework of piDMD.
We consider a range of applications in section \ref{Sec:examples}, with a focus
on shift-invariant, conservative, self-adjoint, local and causal systems.
Section \ref{Sec:conclusion} concludes with a discussion of the
limitations of piDMD and suggests several future research directions.
An open source implementation of piDMD with support for over 30 matrix manifolds 
is available in \textsc{Matlab} at \url{www.github.com/baddoo/piDMD}.
\section{Mathematical background}
\label{Sec:background}
In this section we provide further details on DMD and Procrustes problems.
Throughout the article, we assume that we have access to $m$ snapshots pairs
of $n$ features each: \mbox{$\{(\bx_j,\, \by_j), \, j = 1,\dots, m\}$}. 
For example, $\bx_j$ may be a discretized fluid flow field at time $t_j$ and $\by_j$ may be the flow field at the next time step $t_{j+1}$.  
It is convenient to arrange the data into $n\times m$ snapshot matrices of the form
%
\begin{align}
\bX&
=  \begin{bmatrix} | & | & | \\ \bx_1& \cdots & \bx_{m} \\ | & | & |\end{bmatrix}
= \begin{bmatrix} 
		\rule[0.5ex]{1.2em}{0.55pt} & \tilde{\bx}_1 & \rule[0.5ex]{1.2em}{0.55pt} \\[-1ex]
		\rule[.8ex]{1.2em}{0.55pt} & {\vdots} & \rule[0.8ex]{1.2em}{0.55pt} \\
		\rule[0.5ex]{1.2em}{0.55pt} & \tilde{\bx}_n & \rule[0.5ex]{1.2em}{0.55pt} \\
	\end{bmatrix}
, \qquad
\bY
=  \begin{bmatrix} | & | & | \\ \by_1& \cdots & \by_{m} \\ | & | & |\end{bmatrix}
= \begin{bmatrix} 
		\rule[0.5ex]{1.2em}{0.55pt} & \tilde{\by}_1 & \rule[0.5ex]{1.2em}{0.55pt} \\[-1ex]
		\rule[.8ex]{1.2em}{0.55pt} & {\vdots} & \rule[0.8ex]{1.2em}{0.55pt} \\
		\rule[0.5ex]{1.2em}{0.55pt} & \tilde{\by}_n & \rule[0.5ex]{1.2em}{0.55pt} \\
	\end{bmatrix}
	\label{Eq:SnapshotMatrices}
\end{align}
so that rows $i$ of $\bX$ and $\bY$
are the measurements of the $i$-th features and columns $j$  are the $j$-th temporal snapshots. Henceforth, we use $\tilde{\cdot}$ to represent a row vector.
\subsection{Dynamic mode decomposition}
\label{Sec:DMD}

DMD was initially proposed as a dimensionality reduction technique
that extracts dominant spatio-temporal coherent structures from high-dimensional time-series data \cite{Schmid2010}.
In particular, DMD identifies the leading-order spatial eigenmodes of the matrix $\bA$ in \eqref{Eq:DMDAMatrix}, along with a linear model for how the amplitudes of these coherent structures evolve in time. 
DMD has been applied to a range of systems, as summarized in the monograph \cite{Kutz2016} and
the recent review by Schmid \cite{Schmid2022}. 

To address the challenges associated with DMD, researchers have derived many variations on the original algorithm, including
sparsity promoting DMD~\cite{Jovanovic2014},
DMD with control~\cite{Proctor2016}, 
noise-robust variants~\cite{Bagheri2014pof,Dawson2016,Hemati2017tcfd,Askham2018}, 
recursive DMD~\cite{Noack2016jfm},
online DMD~\cite{hemati2014pof,Zhang2019}, 
and versions for under-resolved data in space or time~\cite{Brunton2015jcd,Gueniat2015pof,Tu2014ef}. 
At present, the most widely used variant is `exact DMD' \cite{Tu2014}, which phrases the DMD solution in terms of the Moore--Penrose pseudoinverse.
This solution produces a more even distribution of errors between the terms.
Due to its simplicity and widespread use, most of the comparisons in this paper are made between exact DMD and piDMD.

It is assumed that each $\bx_j$ and $\by_j$ are connected by an unknown dynamical system of the form $\by_j = \bF(\bx_j)$; for discrete-time dynamics $\by_j = \bx_{j+1}$ and for continuous-time dynamics $\by_j = \dot{\bx}_j$.
DMD aims to learn the dominant behaviour of $\bF$ by finding the best
linear approximation for $\bF$ given the data and then performing 
diagnostics on that approximation.
Thus, DMD seeks the linear operator $\bA$ that best maps the snapshots in the set $\{ \bx_j \}$ to those in the set $\{ \by_j \}$:
\begin{align}
\by_j \approx \boldsymbol{A}\bx_j \qquad \qquad \textrm{for } j = 1, \dots, m. 
	\label{Eq:dmd1}
\end{align}
Expressed in terms of the snapshot matrices in \eqref{Eq:SnapshotMatrices}, the linear system in \eqref{Eq:dmd1} becomes
\begin{align}
{\boldsymbol{Y} \approx \boldsymbol{A}\boldsymbol{X}},
\end{align} 
and the optimization problem for $\bA$ is given by \eqref{Eq:DMD}.
The minimum-norm solution for $\bA$ is given by
\begin{equation}
\bA = 
\bY \bX^{\dagger} = \bY \bV \bSigma^\dagger \bU^\ast,
\label{Eq:exactDMD}
\end{equation}
where $\dagger$ indicates the Moore--Penrose pseudoinverse \citep{Golub2013} and $\boldsymbol{X}=\boldsymbol{U\Sigma V}^\ast$ is the singular value decomposition.
%
In many applications, the state dimension $n$ is very large, and
forming or storing $\bA$ explicitly becomes impractical.
Instead, we use a rank-$r$ approximation for $\bA$, denoted by $\hat{\bA}$, where $r \ll n$.
To form $\hat{\bA}$, we construct the optimal rank-$r$ approximation for $\bX$ using the truncated singular value decomposition~\citep{Eckart1936}:
$\bX\approx\bU_r\bSigma_r \bV_r^\ast$.
We then project $\bA$ onto the leading $r$ principal components of $\bX$ as
\begin{equation}
\hat{\bA} = \bU_r^\ast \bA \bU_r 
= \bU_r^\ast \bY \bV_r \bSigma_r^{-1}.
\end{equation}
It is now computationally viable to compute the eigendecomposition of $\hat{\bA}$ as
\begin{equation}
\hat{\bA} \hat{\bPsi} = \hat{\bPsi} \boldsymbol{\Lambda}.
\end{equation}
The eigenvectors of $\bA$ can be approximated from the reduced eigenvectors $\bPsi$ by \cite{Tu2014}
\begin{equation}
\bPsi = \bY \bV \bSigma^{-1} \hat{\bPsi}.
\end{equation}
This eigendecomposition is connected to the Koopman operator
of the system, and allows reconstructions and predictions \cite{Rowley2009,Mezic2013,Kutz2016,Brunton2021}.
For example, for a discrete-time system (i.e., $\by_k = \bx_{k+1}$) with evenly spaced samples in time, then the eigenvectors form a linearly independent set and
\begin{align}
\bx_j = \bPsi \bLambda^{j-1} \bb
\end{align}
where the vector $\bb$ contains the weights of the modes in the initial condition: $\bb = \bPsi^{\dagger}\bx_1$.
From the above, it is clear that the eigenvalues $\bLambda$ govern the temporal
behaviour of the system and the eigenvectors $\bPsi$ are the spatial modes.

\subsection{Procrustes problems}
\label{Sec:procrustes}
Procrustes problems \eqref{Eq:procrustes} comprise of finding the optimal transformation between two matrices subject to certain constraints on the class of admissible transformations \cite{Schonemann1966,Gower2007,Higham1988}.
According to Greek mythology, Procrustes was a bandit who would
stretch or amputate the limbs of his victims to force
them to fit onto his bed.
Herein, $\bX$ plays the role of Procrustes' victim, $\bY$ is the bed, and $\bA$ is the `treatment' (stretching or amputation).%
\footnote{This terminology was first introduced by Hurley and Cattell \cite{Hurley1962}.}
Procrustes problems \eqref{Eq:procrustes} seek to learn the treatment $\bA$ that best represents the data measurements $\bX$ and $\bY$.
The minimization is usually phrased in the Frobenius norm.
Procrustes analysis finds relevance in a wealth of fields including sensory analysis \cite{Dijksterhuis1994},
data visualization \cite{Cox2008},
neural networks \cite{Zhang2016},
climate science \cite{Richman1986}, and
solid mechanics \cite{Higham1988}.
A summary is available in the monograph \cite{Gower2007}.

The earliest and most common formulation is the 
`orthogonal Procrustes problem' \cite{Schonemann1966,Gower2007}.
Suppose that we have two sets of measurements $\bX$ and $\bY$ that 
we know are related by an unitary (orthogonal) transformation (i.e., a rotation or reflection).
The goal is to learn the best unitary transformation that relates
the data measurements.
Thus, $\bA$ is constrained to be a unitary matrix and the minimization problem is
\begin{flalign}
	\argmin_{\bA^\ast \bA = \bI} \|\bY - \bA \bX \|_F. &
	\label{Eq:orthProc}
\end{flalign}
The solution to \eqref{Eq:orthProc} was derived by Sch\"onemann \cite{Schonemann1966} as
\begin{align}
	\bA = \bU_{YX} \bV_{YX}^\ast 
	\label{Eq:orthSol}
\end{align}
where $\bU_{YX} \bSigma_{YX} \bV_{YX}^\ast = \bY \bX^\ast$ is a full singular value decomposition.
Alternatively, $\bA = \bU_P$, where $\bU_{P} \bH_P = \bY \bX^\ast$ is 
a polar decomposition.
The solution is unique if and only if $\bY \bX^\ast$ is full-rank.
As we show in section \ref{Sec:conservative}, a unitary matrix 
corresponds to an energy-preserving operator.

There are many solutions for Procrustes problems available in the literature
for different matrix constraints \cite{Higham1988}.
When exact solutions are not possible, algorithmic solutions can be effective \cite{Boumal2014}.

%
%
%
%
%
%
\section{Physics-informed dynamic mode decomposition (piDMD)}
\label{Sec:piDMD}


Incorporating physics into ML algorithms involves supplementing an existing technique with additional biases.
Usually, ML practitioners take one of three approaches \cite{Karniadakis2021}.
First, observational biases can be embedded through data augmentation techniques;
however, augmented data is not always available, and incorporating additional training
samples can become computationally expensive.
Second, physical constraints can be included by suitably penalising the loss function~\cite{Wang2020}.
Third, inductive biases can be incorporated directly into 
the machine learning architecture in the form of mathematical constraints~\cite{Loiseau2018}.
As noted by \cite{Karniadakis2021}, this third approach is arguably the most principled method
since it produces model that strictly satisfy the physical constraints.
piDMD falls into the final category.

piDMD incorporates physical principles by constraining the matrix manifold of the 
DMD regression problem \eqref{Eq:procrustes}.
We abandon the low-rank representation of $\bA$ typically sought 
in model order reduction in favour of alternative or additional matrix 
structures more relevant to the physical problem at hand.
In figure \ref{Fig:matrixPictures} we illustrate the matrix structures 
used in this paper, along with their corresponding physical principle and references
for the optimal solutions.

\begin{figure}[t]
	\centering
	\includegraphics[width=0.8\linewidth]{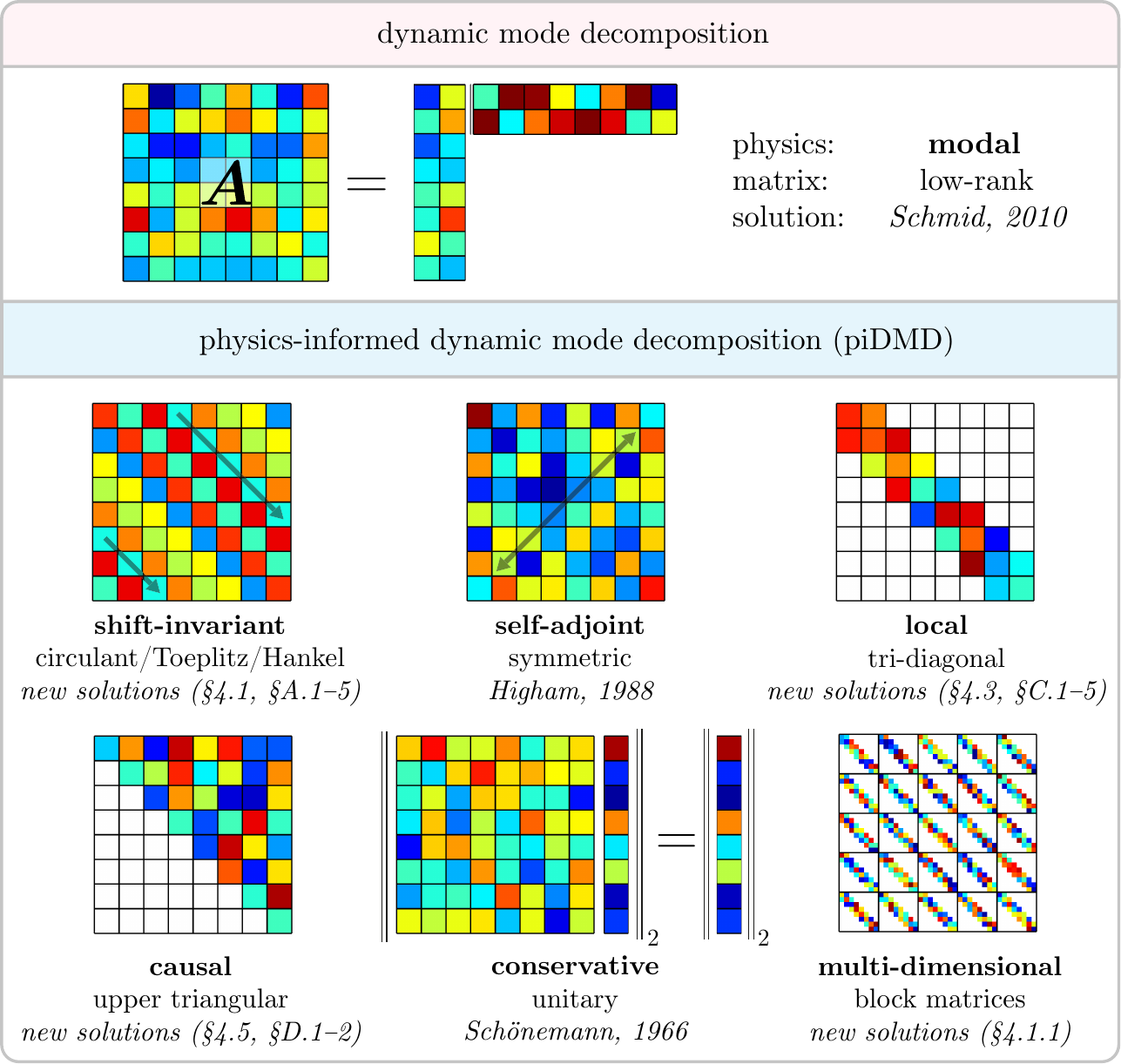}
	\caption{Visual illustrations of the types of matrices, their 
	corresponding physical principles and references to the solutions of the
corresponding optimization problem \eqref{Eq:procrustes}.}
\label{Fig:matrixPictures}
\end{figure}

Analysing a system with the piDMD framework requires four steps: modeling, interpretation, optimization, and diagnostics:
\begin{enumerate}
\item \myuline{Modeling}:
When studying a system from a data-driven perspective, it is common to have
some partial knowledge of the underlying physics of the system.
This first step in piDMD asks the user to summarise the known or suspected physical properties of the system at hand.
\item \myuline{Interpretation:}
Having determined the physical principle we would like to enforce, we must translate
these laws into the matrix manifold to which the linear model should be constrained.
This step involves knowledge of the data collection procedure such as the spatial grid or the energy inner product.
\item \myuline{Optimization:}
Equipped with a target matrix manifold, we may now solve the relevant Procrustes problem \eqref{Eq:procrustes}.
By suitably modifying the optimization routine, we can guarantee that the resulting
model satisfies the physical principle identified in step 1.
\item \myuline{Diagnostics:}
The final step of piDMD involves extracting physical information from the learned model $\bA$.
For example, one may wish to analyse the spectrum or modes, compute the resolvent modes~\cite{herrmann2021jfm}, perform predictions,
or investigate other diagnostics.
\end{enumerate}

As an example, consider the travelling wave solution explored in the introduction (figure \ref{Fig:schematic}).
We began with the physical principle that the system is shift invariant (`modeling'), and this 
lead us to seek a circulant matrix model (`interpretation').
We found a solution to the corresponding Procrustes problem (`optimization')
and investigated the predictive behaviour of the model (`diagnostics').

While each of the four steps can be problematic,
the optimization step is arguably the most conceptually and computationally difficult.
In particular, finding a solution for the desired Procrustes problem can be challenging. Moreover, any solutions are quite peculiar to the matrix manifold under consideration: the 
symmetric Procrustes problem is largely unconnected to, say, the tridiagonal Procrustes problem. 
Thus, in this article, we present many new solutions for \eqref{Eq:procrustes} with different physics-informed matrix manifold constraints.
When exact solutions are not possible (e.g., if the manifold constraint is quite complicated)
there are sophisticated algorithmic solutions available \cite{Boumal2014}.

Standard DMD exploits the low-rank structure of $\bA$ to efficiently perform diagnostics on the learned model.
Some of the manifolds we consider (such as circulant, tridiagonal, upper triangular) do not have an obvious
or useful low-rank approximation.
Instead, these matrices often have an alternative structure that can be exploited to perform fast diagnostic tests.
For example, tridiagonal matrices rarely have a meaningful low-rank approximation, but nevertheless admit a fast
eigendecomposition and singular value decomposition \cite{Golub2013}.

Some of the matrix manifolds we consider
(such as symmetric, triangular, tridiagonal or circulant)
can be phrased as linear equality constraints and can, in principle, be solved
with linear-equality constrained least squares \cite{Golub2013}.
However, the number of equality constraints needed is $\mathcal{O}(n^2)$
so the resulting least squares matrix will have $\mathcal{O}(n^2)$ rows,
which is intractably large in most applications.
Thus, we avoid phrasing the piDMD constraints in terms of linear equality constraints and instead 
exploit properties of the matrix manifold to find solutions that can be efficiently implemented
(see sections \ref{Sec:selfAdjoint} and \ref{Sec:shiftInvariant}, for example).
\section{Applying piDMD to enforce canonical physical structures}
\label{Sec:examples}
We now present a range of examples of piDMD.
Each of the following sections investigates a detailed application of piDMD for a specific physical principle.
A summary is illustrated in figure \ref{Fig:grid}.
The physical principles, corresponding matrix structures, and optimal solutions and extensions are summarised below:
\begin{enumerate}
  \setlength\itemsep{0.3em}
    \item[\S \ref{Sec:shiftInvariant}:]
    \myuline{Shift-invariant:}
    {circulant (and symmetric, skew-symmetric, or unitary, \S \ref{Ap:circulantSymmetric}), low rank (\S \ref{Ap:circulantLowRank}), non-equally spaced samples (\S \ref{Ap:nufft}), total least squares (\S \ref{Ap:circulantTLS}), Toeplitz \& Hankel (\S \ref{Ap:toeplitz}})
    \item[\S \ref{Sec:conservative}:] 
    \myuline{Conservative:}
    {unitary (\S \ref{Sec:procrustes})}
    \item[\S \ref{Sec:selfAdjoint}:] 
    \myuline{Self-adjoint:}
    {symmetric (\S \ref{Sec:selfAdjoint}), skew-symmetric (\S \ref{Ap:skewSymmetric}), symmetric in a subspace (\S \ref{Sec:schrodinger})}
    \item[\S \ref{Sec:local}:]
    \myuline{Local:}
    {tridiagonal (\S \ref{Sec:local}), variable diagonals (\S \ref{Ap:localVar}), periodic (\S \ref{Ap:localPer}), symmetric tridiagonal (\S \ref{Ap:localSym}),  total least squares (\S \ref{Ap:localTLS}), regularized locality (\S \ref{Ap:localGen})}.
    \item[\S \ref{Ap:triangular}:]
    \myuline{Causal:} {upper triangular (\S \ref{Ap:triangularUpdate} \ref{Ap:triangularStable})}
\end{enumerate}
\begin{figure}[t]
	\centering
	\includegraphics[width=\linewidth]{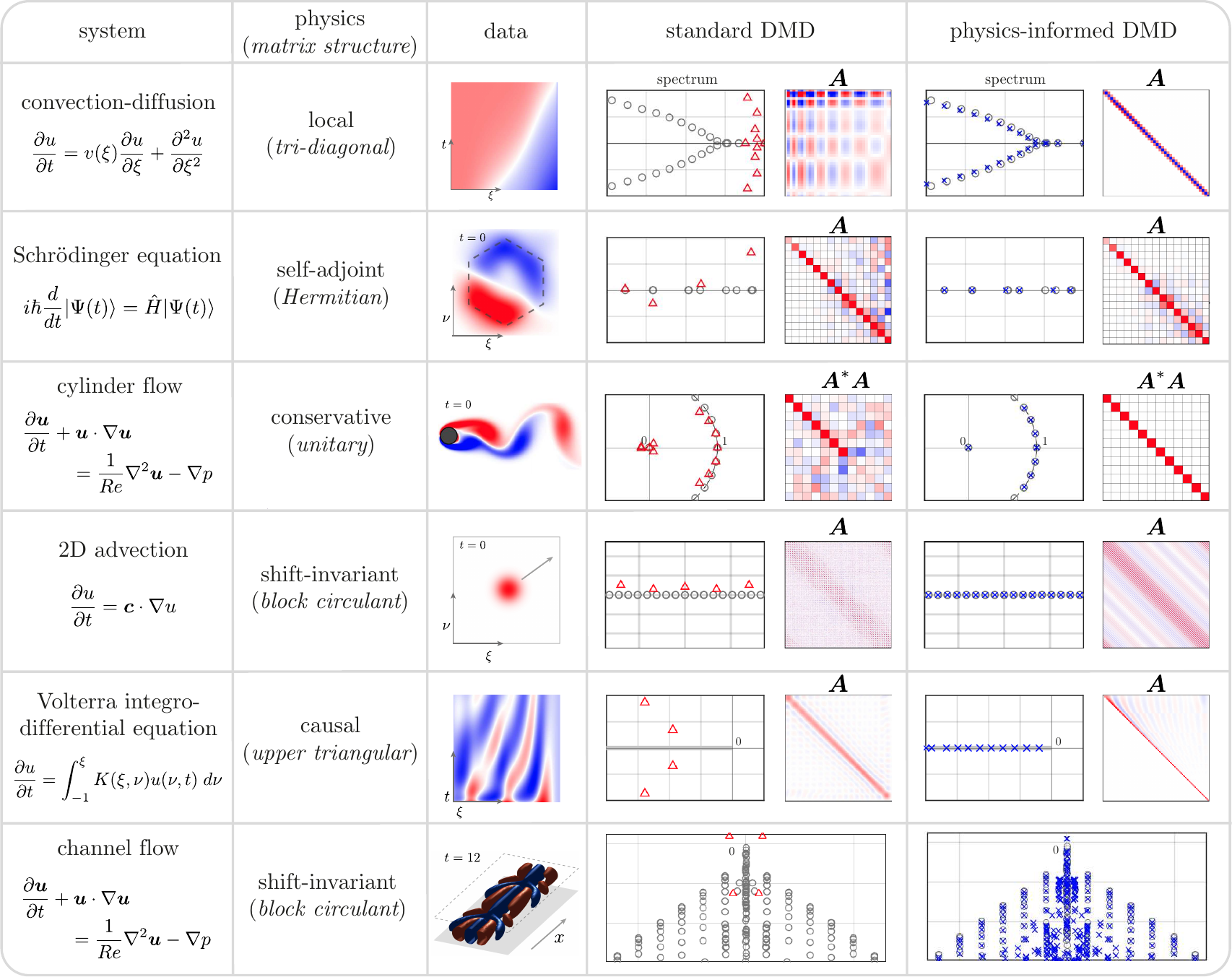}
	\caption{
	A comparison of the models learned by exact DMD and piDMD for a range of applications.
	piDMD identifies the spectrum of the true operator with higher accuracy than exact DMD.
	The structure of the model matrices are also illustrated; piDMD models are generally more coherent than those learned by exact DMD.
	Details are given in the corresponding sections.
In the spectrum subplots, we plot the true eigenvalues as 
{\protect\tikz \protect\draw[draw = gray,line width = 1pt] (.5,0) circle (3pt);},
the DMD eigenvalues as
{\protect\tikz \protect\draw[draw=red,thick] (0,0) --
	(0.2cm,0) -- (0.1cm,0.2cm) -- (0,0);}
, and the piDMD eigenvalues as 
{\protect\tikz \protect\draw[draw = blue, thick] (.5,0) node[blue, cross=3.5pt] {};}.
}
	\label{Fig:grid}
\end{figure}

\subsection{Shift-invariant systems}
\label{Sec:shiftInvariant}
Shift invariance (also called `translation invariance') is ubiquitous in the physical sciences.
Spatially homogeneous systems -- such as constant coefficient PDEs or convolution operators --
are shift invariant, and thus appear identical with respect to any (stationary) observer.
In view of Noether's theorem \cite{noether1918invariante}, the conserved quantity related to shift
invariance is linear momentum; as such, incorporating shift invariance into a DMD model means that the model preserves linear momentum.
While it may be unusual for practical engineering problems to have a truly homogeneous 
direction, many problems of basic physical interest -- such as wall-bounded turbulent flow --
frequently possess a spatially homogeneous dimension.
In the sequel, we derive a piDMD formulation for shift-invariant operators.

We define $\mathcal{S}_\phi$ as the $\phi$-shift operator
i.e. $\mathcal{S}_{\phi} v(\xi) = v(\xi + \phi)$ for all
test functions $v$.
We say that a space-continuous linear operator $\mathcal{A}$
is shift invariant if $\mathcal{A}$ commutes with the $\phi$-shift operator for all shifts $\phi$:
\begin{align}
	\mathcal{S}_\phi \mathcal{A} = \mathcal{A} \mathcal{S}_\phi.
	\label{Eq:shift}
\end{align}
An application of the shift operator shows that if $\mathcal{A}$ is shift invariant then the quantity
$\e^{-\lambda \xi} \mathcal{A} \e^{\lambda \xi}$ is constant for all $\xi$.
Thus, neglecting boundary conditions for now, $\{\e^{\lambda \xi}\}$ are eigenfunctions of $\mathcal{A}$.
If the domain is normalised to $[-1,1]$ and has periodic boundaries, then
 $\lambda = l \pi i$ for integer $l$.
In this case the corresponding eigenfunctions form an orthogonal basis and the operator is diagonalised by the known eigenfunctions.

In the following, we assume that we are studying a shift-invariant
system on a periodic domain $\xi \in [-1,1]$.
We move from the continuous formulation to a discretized space and assume that we are
studying a function $u(\xi,t)$
and have access to evenly spaced samples of $u$ at $\bxi =[ -1,\, -1+\Delta \xi, \dots, 1-\Delta \xi ]$.
We discuss the case of non-equally spaced samples in appendix \ref{Ap:nufft}.
If we define the state variable as $\bx(t) = \left[u(\xi_1,t), \,u(\xi_2,t), \cdots \,, u(\xi_n, t) \right]$
then the discrete-space linear operator $\bA$ that generates $\bx(t)$ is diagonalized by its eigenvectors:
\begin{align}
	\bA = \bFc \diag\left( \hat{\ba} \right) \bFc^{-1}
	\label{Eq:circulantDiag}
\end{align}
where $\bFc_{j,k} = \e^{2 \pi \i (j-1) (k-1)/n }/\sqrt{n}$ and $\bFc^{-1} = \bFc^\ast$ and $\{\hat{a}_j \}$ are the unknown eigenvalues.
Equation \eqref{Eq:circulantDiag} is equivalent to stating that $\bA$ is circulant:
\begin{align}
\bA=
\begin{bmatrix}
a_{0}&a_{n-1}&\dots &a_{1}\\
a_{1}&a_{0}&\ddots&\vdots\\
\vdots &\ddots&\ddots &a_{n-1} \\
a_{n-1}&\dots &a_{1}&a_{0}
\end{bmatrix},
\end{align}
i.e. $\bA_{j,k}= a_{(j-k)\, \textrm{mod}\, n}$.
If the boundary conditions are instead Dirichlet or Neumann then $\bA$ is not circulant but Toeplitz; we solve the corresponding Procrustes problem in \S \ref{Ap:toeplitz}.

Substituting \eqref{Eq:circulantDiag} into \eqref{Eq:procrustes} and 
noting that the Frobenius norm is invariant to unitary transformations allows
\eqref{Eq:procrustes} to be transformed to
\begin{align}
	\argmin_{\hat{\ba}} \left\|\diag(\hat{\ba}) \bXc - \bYc \right\|_F ,
	\label{Eq:circulant2}
\end{align}
where $\bXc = \bFc^\ast \bX$ and $\bYc=\bFc^\ast \bY$ are the spatial discrete Fourier transforms of $\bX$ and $\bY$.
As such, $\bXc$ and $\bYc$ can be efficiently formed in $O(m n\log(n))$ operations using Fast Fourier Transform (FFT) methods
\cite{Dutt1993}.

The rows of the cost function \eqref{Eq:circulant2} now decouple to
produce $n$ minimization problems:
\begin{align}
	\argmin_{\hat{a}_j} \|\hat{a}_j \tilde{\bXc}_j - \tilde{\bYc}_j \|_F
	\qquad \textnormal{for } 1 \leq j \leq n,
	\label{Eq:circulant2b}
\end{align}
where $\tilde{\bXc}_j$ and $\tilde{\bYc}_j$ are the $j$th rows of $\bXc$ and $\bYc$,
respectively.
An alternative view of this step is that spatial wavenumbers decouple in shift-invariant systems,
so we may analyse each wavenumber individually.
The optimal solution for each eigenvalue follows as
\begin{align}
	\hat{a}_{j} = \tilde{\bYc}_j \tilde{\bXc}_j^\dagger
	=\left.{\tilde{\bYc}_j\tilde{\bXc}_j^\ast }\middle/{\|\tilde{\bXc}_j \|^2_2}\right. .
	\label{Eq:circulantSol}
\end{align}
In principle, standard DMD methods \cite{Schmid2010,Tu2014} could be encouraged to respect shift invariance by augmenting the data
matrices with their shifted counterparts. However, one would strictly need to include $n$ shifts,
thus producing $nm$ samples which is usually too large in applications.
In contrast, the above solution is extremely efficient via the use of the FFT and the decoupling of wavenumbers.

The appendices include alternative solutions for cases where the system is shift-invariant and
symmetric, skew-symmetric or unitary (\S \ref{Ap:circulantSymmetric}),
low rank (\S \ref{Ap:circulantLowRank}),
the samples are not equally spaced (\S \ref{Ap:nufft}),
the total least squares case (\S \ref{Ap:circulantTLS}),
and when the system is Toeplitz or Hankel (\S \ref{Ap:toeplitz}).

\subsubsection{Example: plane channel flow}
Our next example is the incompressible flow inside a plane channel of size $2\pi\times 2\times 2\pi$ along the spatial $x$, $y$, and $z$ coordinates that indicate the streamwise, wall-normal, and spanwise directions, respectively. The configuration considers a Reynolds number of $\Rey=2000$ based on the channel half-height and the centerline velocity, and periodic boundary conditions in the open faces of the channel, hence the flow is homogeneous in the $x$ and $z$ directions. We use the spectral code \emph{Channelflow} \cite{Gibson2014} to perform direct numerical simulations (DNS) with the same numerical configurations as that presented in \cite{herrmann2021jfm}.

The dataset investigated is generated from a DNS of the response of the laminar flow with parabolic velocity profile to a localized perturbation in the wall-normal velocity component of the form
\begin{equation}
v(x,y,z,0)=\left(1-\frac{r^2}{c_r^2}\right)\left(\cos(\pi y) +1\right)e^{\left(-r^2/c_r^2-y^2/c_y^2 \right)},
\end{equation}
where $r^2=(x-\pi)^2+(z-\pi)^2$, the parameters are set to $c_r=0.7$ and $c_y=0.6$, and the amplitude of the perturbation was scaled to have an energy-norm of $10^{-5}$ to ensure that the effect of nonlinearity is negligible. This initial condition, which was first studied in \cite{ilak2008pof} and also in \cite{herrmann2021jfm}, is a model of a disturbance that could be generated in experiments using a spanwise and streamwise periodic array of axisymmetric jets injecting fluid perpendicular to the wall. The dataset obtained from this simulation is then comprised of a sequence of $300$ snapshots of the three-dimensional velocity-perturbation field recorded every $0.5$ time units.

\begin{figure}[t]
	\centering
	\includegraphics[width=.9\linewidth]{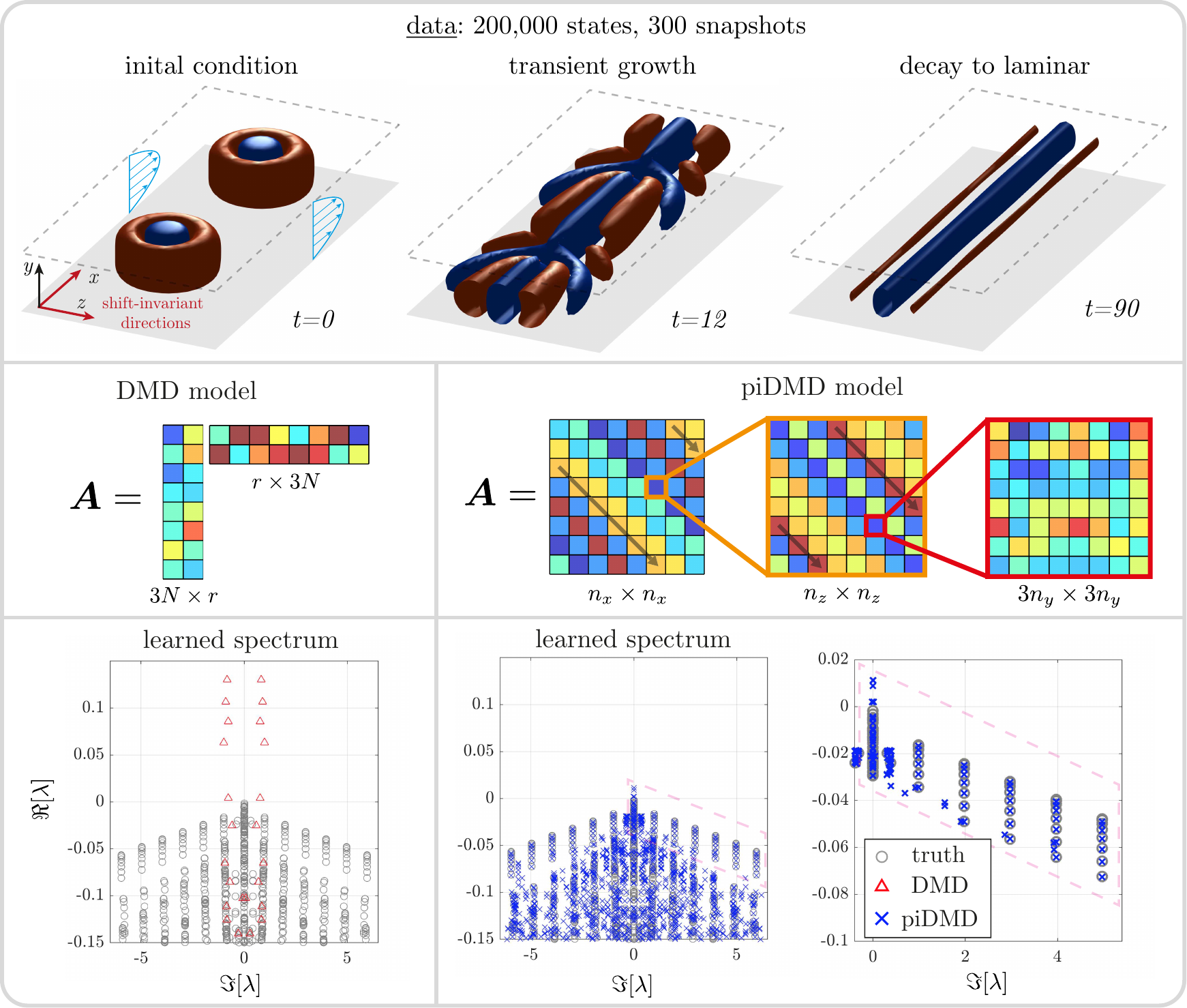}
	\caption{Learning the spectrum of the linearized Navier--Stokes equations
from velocity measurements of the response to a localized disturbance in a channel flow.
piDMD embeds the shift-invariant structure of the Navier--Stokes equations into the
model learning framework, thus learning the spectrum of the linearized operator
with improved accuracy over exact DMD.
	}%
	\label{Fig:channelflow}
\end{figure}

Our aim is to learn the spectrum of the underlying linear operator that \emph{best} describes the evolution of the DNS snapshots. When directly applying exact DMD to this dataset, the algorithm attempts to find global modes for the three-dimensional velocity field and their spectrum. This is challenging because modes associated to different streamwise and spanwise wavenumbers may be mixed, leading to spurious eigenvalues, as shown in figure~\ref{Fig:channelflow}.
However, we know that the flow is homogeneous in the $x$ and $z$ coordinates, and is therefore shift-invariant in these directions. Hence, we can leverage piDMD to incorporate this property, forcing the resulting matrix to respect, by construction, the structure with three nested levels shown in figure~\ref{Fig:channelflow}.
In practice, this amounts to reshaping every snapshot, taking the FFT in the dimensions corresponding $x$ and $z$, and performing DMD on the $y$-dependent Fourier amplitudes for every streamwise and spanwise wavenumber tuple. Although simple, this has a tremendous impact on the quality of the learned spectrum. Shift-invariant piDMD produces modes that are associated with a single wavenumber tuple, which forces the modes to respect the spatial periodicity of the underlying system.

The spectra learned from both approaches are compared in figure~\ref{Fig:channelflow}, showing that piDMD results in a far more accurate eigenvalue spectrum. The DMD calculation considers a truncated SVD of rank $200$. The piDMD results were obtained with a rank truncation of $50$ at every wavenumber tuple.
\subsection{Conservative systems}
\label{Sec:conservative}

Conservation laws are foundational to all areas of science.
Identifying quantities that remain constant in time --
such as mass, momentum, energy, electric charge, and probability
-- allow us to derive universal governing equations, apply inductive reasoning, 
and understand basic physical mechanisms.
In this section we demonstrate how conservation laws can be incorporated into the
DMD framework.

Suppose that we are studying a system that we know conserves energy.
In applications of DMD, it is implicitly assumed that measurements of the 
state have been suitably weighted so that the 
square of the 2-norm corresponds to the energy of the state~\cite{herrmann2021jfm}: \mbox{$E(\bx)  = \| \bx \|_2^2$}.
In these variables, 
the original optimization problems~(\ref{Eq:DMD}, \ref{Eq:procrustes})
equate to finding the model $\bA$ that minimises the energy of the error between
the true and predicted states ($\by_k$ and $\bA \bx_k$ respectively).
Thus, if $\bA$ represents a discrete-time linear dynamical system 
($\by_k = \bx_{k+1}= \bA \bx_k$)
then $\bA$ is an energy preserving operator if and only if
\begin{align}
	E(\bA \bx)=	\| \bA \bx \|_2^2 = \|\bx\|_2^2 = E(\bx) \qquad \textnormal{for all } \bx \in \mathbb{R}^n.
	\label{Eq:orth}
\end{align}
In words, \eqref{Eq:orth} states that $\bA$ does not change the
energy of the system but merely redistributes energy between the states.
In matrix terminology, \eqref{Eq:orth} holds if and only if $\bA$ is \emph{unitary}.
Therefore, for conservative systems the optimization problem \eqref{Eq:procrustes} is the orthogonal Procrustes problem described in section \ref{Sec:procrustes},
and the solution is given by \eqref{Eq:orthSol}.
The eigenvalues of $\bA$ lie on the unit circle, and the eigenvectors
are orthogonal.
Thus, the eigenvectors oscillate in time, with no growth or decay.
Since the solution of the orthogonal Procrustes problem \eqref{Eq:orthSol}
requires the full SVD of an $n\times n$ matrix, it can be more computationally 
efficient to first project onto the leading POD modes and build a model therein.
In this case, the model is only energy preserving within the subspace spanned by the leading POD modes.

Noise is a substantial problem for most DMD methods, and a common
remedy is to phrase the DMD optimization \eqref{Eq:DMD} as a total least squares problem \cite{Dawson2016}.
Perhaps surprisingly, the solution to the orthogonal Procrustes problem \eqref{Eq:orthSol} is 
also the solution to the total least squares problem when the solution
is constrained to be orthogonal \cite{Arun1992}.
Thus, the solution \eqref{Eq:orthSol} is optimal even when there is noise in both $\bX$ and $\bY$, as long as the noise has the same distribution in $\bX$ and $\bY$.
There are various problems closely related to the orthogonal Procrustes problem,
including the case where $\bA$ lies on the Steifel manifold (i.e., when $\bA$ is rectangular with orthonormal columns \cite{Elden1999}),
there is a weighting matrix \cite{Viklands2006},
missing data \cite{TenBerge1993}, or high amounts of noise \cite{Pumir2021}.

\subsubsection{Example: \texorpdfstring{$Re = 100$}~~flow past a cylinder}
\label{Sec:cylinder}

\begin{figure}[t]
	\centering
	\includegraphics[width=.85\linewidth]{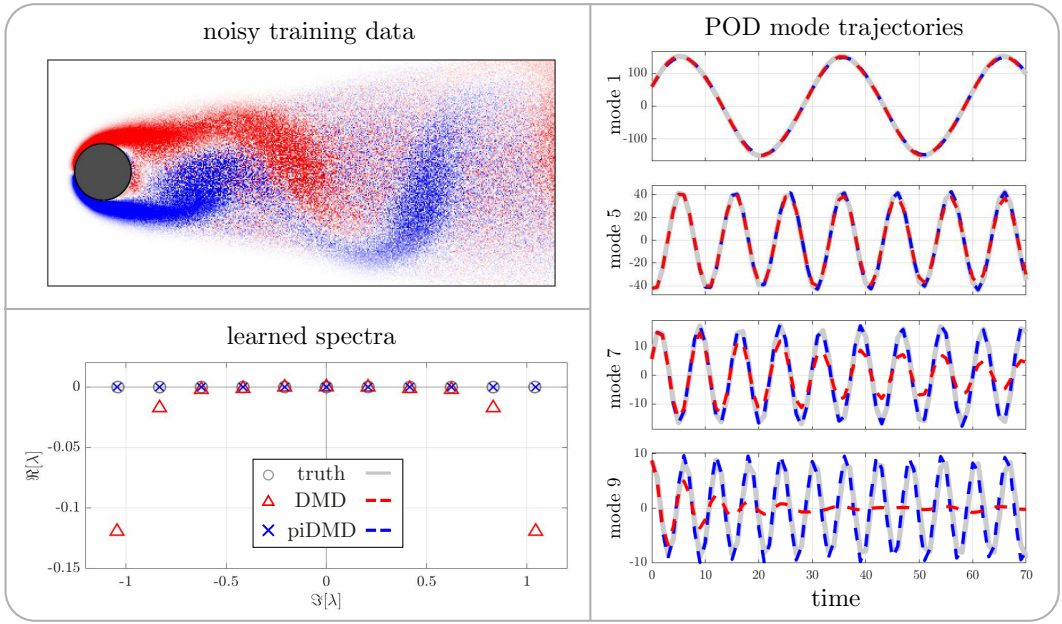}
	\caption{Learning the spectrum of the flow past a cylinder
	with noisy measurements with DMD and energy-preserving piDMD.}%
	\label{Fig:cylinder}
\end{figure}
We now apply this energy-preserving piDMD to study the flow
past a cylinder, which is a benchmark problem for modal decompositions \cite{Kutz2016}.
The Reynolds number is $Re = D U/\nu = 100$
where $D$ is the diameter of the cylinder, $U$ is the free-stream velocity and $\nu$
is the kinematic viscosity.
The data consist of 151 samples of vorticity measurements, corresponding to five periods of vortex shedding, taken on a grid of 199$\times$449 points.
We note that the 2-norm of the measurements is approximately constant in time.
The data are contaminated with 20\% Gaussian noise, as illustrated in the 
top left panel of figure \ref{Fig:cylinder}.
We truncate the data to the first 15 POD modes and learn a standard DMD model,
and an energy-preserving piDMD model.

piDMD learns a more accurate representation of the leading-order dynamics than standard DMD.
For the high-order modes, the eigenvalues learned by DMD exhibit spurious damping 
whereas the eigenvalues learned by piDMD are purely oscillatory and remain on the imaginary axis.
The trajectories of the POD coefficients are also more accurate for piDMD,
and do not exhibit the spurious energy loss caused by noise.

\subsection{Self-adjoint systems}
\label{Sec:selfAdjoint}

Self-adjoint systems are another important class of linear systems
that arise frequently in solid mechanics, heat and mass transfer, fluid mechanics,
and quantum mechanics.
When studying a system known to be self-adjoint
we can use piDMD to constrain our model $\bA$ to lie in the manifold of symmetric (Hermitian)
matrices, such that $\bA = \bA^\ast$.
Symmetric matrices have real eigenvalues and, by the spectral theorem, are diagonalisable.
This restriction places a significant constraint on the learned model
and can substantially reduce the sensitivity of DMD to noise.

The minimum-norm solution of the symmetric Procrustes problem~\cite{Higham2008} is 
\begin{align}
\bA = \bU_X \, \bL \, \bU_X^\ast,
\label{Eq:sym1}
\end{align}
where the entries of $\bL$ are

\begin{align}
	\bL_{i,j} = \overline{\bL_{j,i}} = \begin{cases}
		\dfrac{\sigma_i \overline{\bC_{j,i}} + \sigma_j \bC_{i,j}}{\sigma_i^2 + \sigma_j^2}
		& \textnormal{if } \sigma_i^2 + \sigma_j^2 \neq 0,\\
		\qquad \qquad 0 & \textnormal{otherwise},
\end{cases}
\label{Eq:symL}
\end{align}
where $\bC = \bU_X^\ast \bY \bV_X$
and $\bX = \bU_X \bSigma_X \bV_X^\ast$ is the SVD of $\bX$, and $\sigma_i$ is the $i$th singular value.
The solution of the skew-symmetric case is similar (\S \ref{Ap:skewSymmetric}),
and algorithmic solutions are available when $\bA$ is instead constrained
to be positive definite \cite{Gillis2018}.
For large scale systems, a low-rank approximation to $\bA$ can
be obtained by truncating the SVD of $\bX$. 
If the SVD is truncated to rank $r$, then $\bL$ is an $r \times r$ matrix.
This low-rank approximation preserves the self-adjointness of the model.

The symmetric Procrustes problem is connected to the orthogonal Procrustes problem (\S \ref{Sec:conservative})
via the fact that every unitary
matrix $\bU$ can be expressed as the exponential of a Hermitian matrix $\bH$
as $\bU = \exp(i \bH)$.

\subsubsection{Learning energy states of the Schr\"odinger equation}
\label{Sec:schrodinger}

Possibly the most famous self-adjoint operator is the quantum Hamiltonian, $\hat{H}$ \cite{Griffiths1995}.
In quantum mechanics, physical quantities of interest, such as position, momentum, energy, and spin, are represented
by self-adjoint linear operators called `observables', of which one example is the Hamiltonian.
The Hamiltonian describes the evolution of the probabilistic wave function via the time-dependent Schr\"odinger equation:
\begin{align}
i\hbar {\frac {d}{dt}}\vert \Psi (t)\rangle ={\hat {H}}\vert \Psi (t)\rangle.
\label{Eq:schrodinger}
\end{align}
Solutions of \eqref{Eq:schrodinger} take the form of an eigenmode expansion
\begin{align}
	\Psi(\boldsymbol{\xi},t) = \sum^{\infty}_{j = 1} 
	\alpha_{j} \, \psi_j(\boldsymbol{\xi})\,
	\e^{-\i E_{j} t/\hbar},
\end{align}
where $E_j$ is an energy level (eigenvalue) of $\hat{H}$ with corresponding eigenstate $\psi_j$,
the coefficients $\alpha_{j}$ are determined by the initial distribution
of the wave function, and $\boldsymbol{\xi}$ is the spatial coordinate.

\begin{figure}[t]
	\centering
	\includegraphics[width=\linewidth]{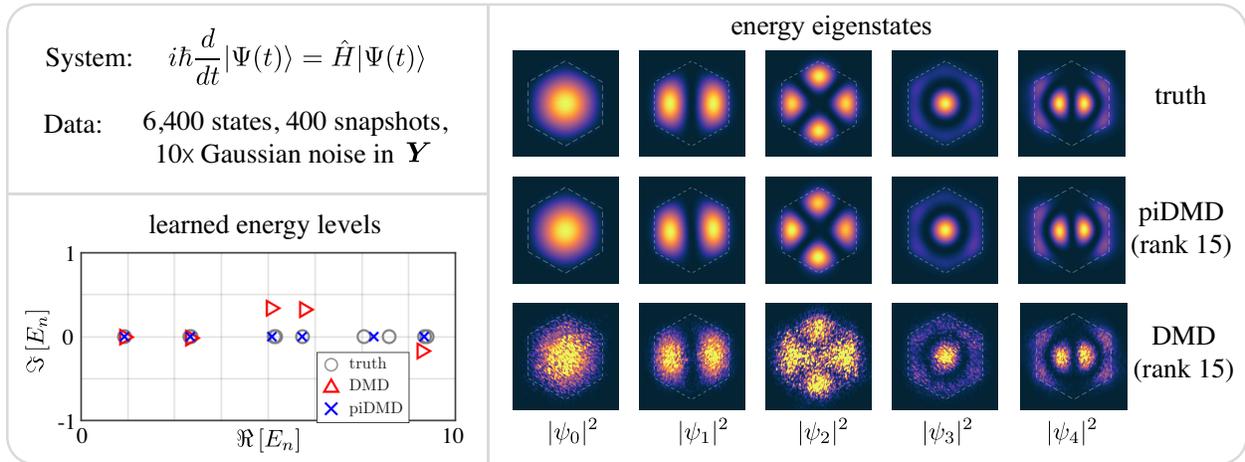}
	\caption{Learning the eigenstates of a quantum Hamiltonian of a hexagonal well with the self-adjoint piDMD.}%
	\label{Fig:Schrodinger}
\end{figure}

We consider a wave function evolving according to the two-dimensional Schr\"odinger's equation
subject to an unknown Hamiltonian function.
The unknown potential function is taken to be a finite well in a hexagonal shape.
We collect measurements of the wave function and its velocity
and train an exact DMD model and a piDMD model.
In practice, empirical measurements of the wavefunction can be obtained through optical 
tomographic methods \cite{Smithey1993,Vogel1989}
or sequential measurements of complementary variables \cite{Lundeen2011}.
The data matrices $\bX$ and $\bY$ are formed from measurements of $\vert \Psi (t)\rangle$ and $i \hbar \frac{d}{dt} \vert \Psi (t)\rangle$ respectively,
which are themselves formed by superposing eigenfunctions obtained by a finite difference method.
Thus, the matrices $\bX$ and $\bY$ are connected by an unknown self-adjoint Hamiltonian $\hat{H}$ and we may apply our self-adjoint DMD optimization.
The measurements of $\bY$ are contaminated with Gaussian noise but
the measurements of $\bX$ are clean.

Constraining the learned Hamiltonian to be self-adjoint reduces the sensitivity of the learning process to noise,
as evidenced by figure \ref{Fig:Schrodinger}.
From the bottom left panel, we see that the energy levels of the piDMD model are physically consistent insofar as they are all real and positive.
In contrast, the energy levels learned by standard DMD exhibit a spurious imaginary component, which will produce
unrealistic growth/decay of the eigenstates.
Both methods miss one eigenstate ($E_n \sim 7.5$) since this particular mode is not well-represented in the data.
Additionally, the eigenstates learned by piDMD are much less noisy than those learned by standard DMD.
This can be explained by noting that the eigenstates learned by piDMD lie in the span of $\bX$, which is clean,
whereas the eigenstates learned by standard DMD lie in the span of $\bY$, which is noisy.
Lemma \ref{NoiseLemma} proves that the symmetric piDMD model \eqref{Eq:sym1} is less sensitive to noise than the exact DMD model \eqref{Eq:exactDMD}.
In summary, a physics-aware approach enables a more accurate identification of quantum energy levels 
and the associated eigenstates.

\subsection{Spatially local systems}
\label{Sec:local}
A hallmark of most physical systems is \emph{spatial locality}.
Points close to one another will usually have a 
stronger coupling than points that are far from one another.
For example, basic physical processes such as convection and diffusion are spatially local.
In most DMD applications, the data are collected from a spatial grid, but the grid rarely plays a role in the DMD analysis beyond forming the energy norm \cite{herrmann2021jfm}. 
In fact, the output of the exact DMD algorithm remains invariant to unitary transformations on the rows and columns of the data matrices, so that randomly shuffling the rows (i.e., the spatial locations) will result in identical DMD models with the corresponding shuffling of the rows of the modes~\cite{Brunton2015jcd,Brunton2019}. 
Knowledge of the underlying grid from which the data were collected 
enables us to bias the learning
process to prioritise relationships that we expect to have a strong local coupling.

Herein we consider the one-dimensional case; the analysis generalises
straightforwardly to higher dimensions.
We assume that the entries of $\bx$ are samples of a function $u$ taken
at $n$ grid points $\{\xi_i\}$.
The grid points can be arranged sequentially so that: $\xi_i < \xi_{i+1}$.
By the spatial locality principle, we expect states that are spatially close to one another to have a stronger
coupling than states that are spatially far from one another.
For example, we may expect entries close to the diagonal of $\bA$ to
be larger than entries far from the diagonal. 
Then the entries of $\bA$ would satisfy
\begin{align}
	|A_{i,j}|\geq|A_{i,k}| \quad \quad \textnormal{for } \quad|j-i|<|k-i|.
	\label{Eq:local1}
\end{align}
Equation \eqref{Eq:local1} is merely a heuristic, and we would expect it
to hold on average rather than for each $(i,j,k)$.
Additionally, \eqref{Eq:local1} is a difficult and expensive condition to 
implement in practice as it involves $\mathcal{O}(n^2)$ inequality
constraints.
Instead, we consider an alternative version of spatial locality
where we only permit coupling between states that are 
sufficiently close:
\begin{align}
	|A_{i,k}|=0 \quad \quad \textnormal{for } \quad d<|k-i|.
	\label{Eq:local2}
\end{align}
Equation \eqref{Eq:local2} describes a $d$-diagonal matrix.
If $d=0$ then states can only affect themselves
and $\bA$ is diagonal (comparable to \eqref{Eq:circulant2}).
A more interesting case is if $d=1$ and we only allow coupling
between adjacent states. Then $\bA$ is tridiagonal:
the entries of $\bA$ are zeros except the leading, upper and lower diagonals:
\begin{align}
	\bA = 
	 \begin{bmatrix}
\beta_1 & \gamma_1 \\
\alpha_2 & \beta_2& \gamma_2 \\
& \alpha_3 & \ddots & \ddots \\
& & \ddots & \ddots & \gamma_{n-1} \\
& & & \alpha_{n} & \beta_n
\end{bmatrix}.
\end{align}
We now solve the optimization problem \eqref{Eq:procrustes} when $\cM$ is the manifold of tridiagonal matrices.
Appendix~\ref{Ap:local} includes many solutions for more general diagonal-like
structures including longer-range and variable coupling (\S~\ref{Ap:localVar}), 
spatially periodic local systems (\S~\ref{Ap:localPer}), 
self-adjoint (\S~\ref{Ap:localSym}), total-least squares (\S~\ref{Ap:localTLS}),
and weaker local structures (\S~\ref{Ap:localGen}).
Since the rows of $\bA$ are decoupled, we can row-wise expand the Frobenius norm
in \eqref{Eq:procrustes} to obtain $n$ smaller minimization problems:
\begin{align}
	\argmin_{\alpha_i,\, \beta_i,\, \gamma_i}
	\|\alpha_i \tilde{\bx}_{i-1} + \beta_i \tilde{\bx}_i + \gamma_i \tilde{\bx}_{i+1} - \tilde{\by}_i \|_2
	\qquad \qquad
	\textnormal{for }
	1\leq i\leq n,
	\label{Eq:tridiagonalMin}
\end{align}
with the convention $\alpha_1 = \gamma_n = 0$ and $\tilde{\bx}_0 = \tilde{\bx}_{n+1} = \boldsymbol{0}$.
Recall that $\tilde{\bx}_i$ and $\tilde{\by}_i$ are the $i$-th rows of $\bX$ and $\bY$ respectively \eqref{Eq:SnapshotMatrices}.
Each minimization problem \eqref{Eq:tridiagonalMin} has the 
minimum-norm solution
\begin{align}
	\begin{bmatrix}
		\alpha_i &
		\beta_i &
		\gamma_i 
	\end{bmatrix}
	= \tilde{\by}_i 
	\begin{bmatrix}
		\tilde{\bx}_{i-1}\\
		\tilde{\bx}_{i}\\
		\tilde{\bx}_{i+1}
	\end{bmatrix}^\dagger
	\qquad\qquad
	\textnormal{for }
	1\leq i\leq n.
	\label{Eq:localSol}
\end{align}
Each of the $n$ minimizations costs $\mathcal{O}(m)$ so the total cost is $\mathcal{O}(mn)$.

\subsubsection{Example: data-driven resolvent analysis}
We consider a system of the form
\begin{align}
	u_t = \mathcal{A} u + f(\xi,t)
\end{align}
and aim to design a control strategy for the forcing $f$ via input-output analysis
\cite{mckeon2010jfm,herrmann2021jfm,jovanovic2021arfm,trefethen1993science}.
Fourier transforming in time and rearranging produces
$\hat{u} = R_\omega \hat{f}$  where $R_\omega = \left(i \omega - \mathcal{A}  \right)^{-1}$
is the resolvent of $\mathcal{A}$
and the hats $\,\hat{\cdot}\,$ denote the Fourier transform in time.
We wish to understand the form of the forcing $\hat{f}$ that
produces the largest response in $\hat{u}$.
Formally, the optimal forcing ($\phi_1$), response ($\psi_1$), 
and gain ($\sigma_1$) for each frequency component satisfy
\begin{align}
	\phi_1 = \argmax_{\|\phi_1 \|_2=1}
\|R_\omega \phi_1 \|_2, \qquad
\sigma_1 = \|R_\omega \phi_1 \|_2, \qquad
\psi_1 = \sigma_1^{-1} R_\omega \phi_1.
\label{Eq:svd}
\end{align}
We may also investigate the higher-order singular triplets
$(\sigma_j, \, \phi_j, \, \psi_j)$ that satisfy equivalent conditions
with the additional requirement of orthogonality of $\{\phi_j\}$ and $\{\psi_j\}$.
Equivalently, we wish to determine the Hilbert--Schmidt decomposition of $R_\omega$
\cite{Kato1966}:
\begin{align}
	R_\omega 
	= \sum_{j = 1}^\infty \sigma_j \psi_j(\xi) \left<\phi_j,\, \cdot \right>,
\end{align}
which is analogous to the SVD of the discretization of $R_\omega$.

Recently, the authors have proposed `data-driven resolvent analysis'
as an efficient method for learning the singular triplets
of a system purely from data measurements \cite{herrmann2021jfm}.
The method is analogous to DMD, except instead of computing
the eigenmodes via eigendecomposition, we compute the resolvent modes
via a singular value decomposition.
In the following example, we demonstrate that piDMD can be integrated
into data-driven resolvent analysis to provide physics-informed diagnostic 
information about the optimal forcings and responses of the system.

We consider the problem of designing a strategy to control the concentration 
of a solute that is governed by unknown dynamics.
The solute concentration $u$ is governed by the convection-diffusion operator
$\mathcal{A}u = u_{\xi \xi} + a(\xi) u_\xi$
where $a$ represents a non-constant convection term (here taken to be a random function) and boundary conditions $u_\xi(-1) = u(1) = 0$.
A typical simulation is illustrated in the left panel
of figure~\ref{Fig:local}.
The convection-diffusion operator is a canonical example of a non-normal system where the typical eigenvalue analysis fails to provide meaningful information \cite{Reddy1994,Trefethen2005} .
Our goal is to determine the forcings $\hat{f}$ that produce
the largest increase in solute in the frequency domain.

\begin{figure}[t]
	\centering
	\includegraphics[width=\linewidth]{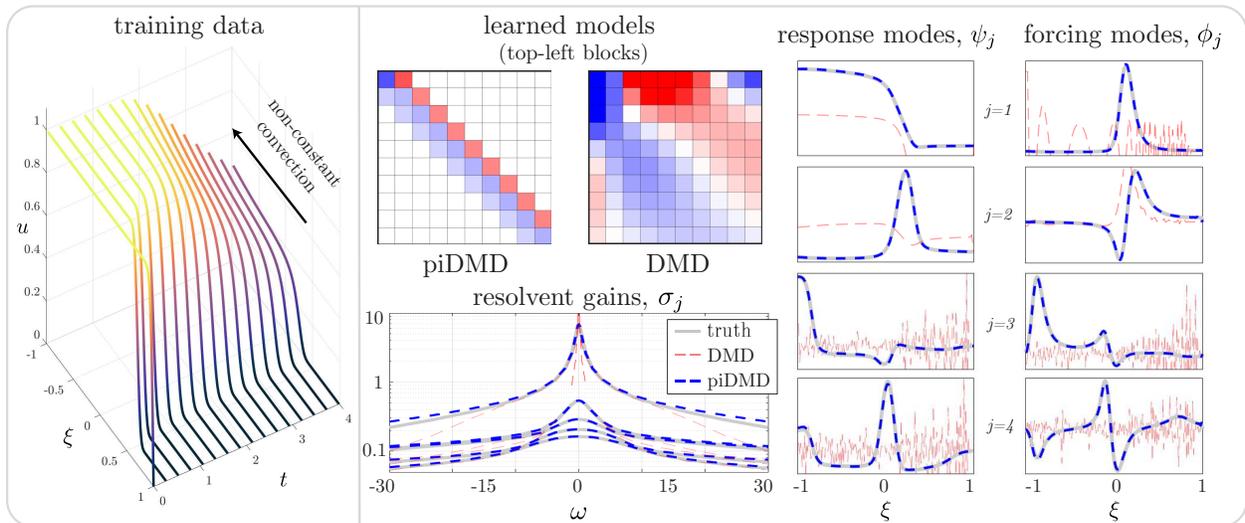}
	\caption{Data-driven resolvent analysis of the convection-diffusion
		equation for DMD and piDMD.
		The piDMD model is constrained to be tri-diagonal and therefore respects the spatial locality of the system.
		Incorporating this physical structure results a significant improvement in the identification of the resolvent norm, response modes, and forcing modes.
	}%
	\label{Fig:local}
\end{figure}

We solve the convection-diffusion equation in Chebfun \cite{Chebfun} and collect 2,000 snapshots of 
$u$ evaluated at 100 evenly spaced grid points.
The snapshots matrices $\bX$ and $\bY$ are formed from measurements of $u$
and $u_t$ respectively. 
We apply data-driven resolvent analysis with DMD and piDMD and compute the learned operators, resolvent norm, and forcing and response modes. 
For both algorithms, we approximate $u_t$ from the measurements of $u$.
The piDMD model is constrained to be tridiagonal, and the DMD model has rank 40.
From the visualizations of the learned models in the center 
of figure~\ref{Fig:local}, 
we see that the standard DMD identifies spurious relationships between the states,
whereas piDMD, by construction, identifies a stronger coupling between
adjacent states.
The model learned by piDMD is reminiscent of the finite difference matrix;
indeed, local piDMD may be viewed as a data-driven discovery of a finite difference stencil.

The benefit of piDMD is evidenced in its accurate identification of
the resolvent modes (for $\omega = 1$), as seen in the right panel of figure~\ref{Fig:local}.
The true resolvent modes are computed in Chebfun (see section 12 of \cite{Aurentz2017}).
The modes learned by standard DMD are highly inaccurate for all $j$, and are completely incoherent
for $j>2$.
In contrast, piDMD accurately determines the mode shapes, and associated
gains for a range of frequencies.
The physics-informed approach even determines the correct boundary conditions of the response and forcing modes, and resolves the complicated structures associated with the non-normality
of the system.
In particular, enforcing spatial locality enables piDMD to 
identify modes that are poorly represented by the data, which allows for 
diagnostic analysis that generalises outside of the training regime.
Equipped with these data-driven modes, one could now design a low-rank control
strategy for the solute concentration.

Analogously to finite difference methods, the accuracy of local piDMD involves a trade-off between the grid sizes in space and time.
Future work should focus on clarifying these issues.
\subsection{Causal systems}
\label{Ap:triangular}
Causality is the process by which the behaviour of one state (the cause) influences another state (the effect).
Many systems posses a spatially causal structure where each state depends only on `upstream' states and
are unaffected by `downstream' states.
Specifically, we may express the $j$th state as a function of the $n-j$ upstream states only:
$y_{j} = f(x_j, \, x_{j+1}, \, \dots \, x_n)$.
For example, $y_n$ depends only on $x_n$ whereas $y_1$ depends only all the elements of $\bx$.
Accordingly, if we know that the system at hand possesses a spatially causal structure
we may seek a piDMD model that is upper triangular.
Such structures are typical of a system with causal features,
for example, time-delay coordinates \cite{Brunton2017}, feed-forward/strict-feedback
systems \cite{Krishnamurthy2007,Teel1996}, and other control systems \cite{Annaswamy1994}.

As in section \ref{Sec:local}, we may decouple the rows of $\bA$ in \eqref{Eq:procrustes}
and expand the cost function row-wise to obtain $n$ smaller minimization problems:
\begin{align}
	\argmin_{\bA_{j,j:n}}
	\|\bA_{j,j:n} \bX_{j:n,\, :} - \tilde{\by}_j \|_2
	\qquad \qquad
	\textnormal{for }
	1\leq j\leq n.
	\label{Eq:triangularMin}
\end{align}
where we have employed the \textsc{Matlab} colon notation ${j\,{:}\,n}$ to indicate entries from $j$ to $n$
and $:$ to indicate entries from $1$ to $n$.
Here, $\tilde{\by}_j$ represents the $j$-th row of $\bY$.
The (minimum norm) solution of \eqref{Eq:triangularMin} is obtained via the pseudoinverse of each sub-block of $\bX$ as
\begin{align}
	\bA_{j,j:n} = \tilde{\by}_j \bX_{j:n,\, :}^\dagger
	\qquad \qquad
	\textnormal{for }
	1\leq j\leq n.
	\label{Eq:triangularSol}
\end{align}
Naively evaluating \eqref{Eq:triangularSol} for each $j$ would require building $n$ pseudoinverses
for a total cost of $\mathcal{O}(n^2 m \min(m,n))$.
Fortunately, consecutive blocks of $\bX$ (e.g. $\bX_{j-1:n,\, :}$ and $\bX_{j:n, \, :}$)
are related by a rank-1 update, so we can solve every problem in \eqref{Eq:triangularSol} 
for a total of $\mathcal{O}(nm \min(n,m))$ operations.
(see appendix \ref{Ap:triangularUpdate} for further details).
However, this approach is very numerically unstable
when applied to realistic data with high condition number.
Thus, in appendix \ref{Ap:triangularStable} we derive a more stable alternative.

We illustrate this causal structure on a simple example defined by
the Volterra-type integro-differential equation
\begin{align}
	\frac{\partial u}{\partial t}(\xi,t)
	= \int_{-1}^{\xi} K(\xi,\nu) u(\nu,t), \d \nu 
	\qquad \qquad -1 \leq \xi \leq 1.
	\label{Eq:intDiff}
\end{align}
Similar systems have been used to model transmission lines in neural networks 
during bursting activity \cite{Jackiewicz2008}
and the spread of disease in epidemics \cite{Medlock2003}.
Taking $K(\xi,\nu) = \sqrt{1-\xi^2} \sqrt{1-\nu^2}$ and $u(\xi,0) = \exp(-\xi^2)$,
we numerically simulate \eqref{Eq:intDiff} and construct $\bX$ and $\bY$ from
measurements of $u$ only (velocity measurements are not used).
The results of piDMD are illustrated in the penultimate row of figure \ref{Fig:grid},
and indicate that piDMD is able to learn the leading eigenvalues while standard DMD fails. Note that the true spectrum here is continuous.

\section{Conclusions and outlook}
\label{Sec:conclusion}
This work presents physics-informed dynamic mode decomposition (piDMD),
a physics-aware modal decomposition technique that extracts coherent structures from high-dimensional time-series data.
Rephrasing the DMD regression \eqref{Eq:DMD} as a Procrustes problem \eqref{Eq:procrustes} shows that partially-known physics can be incorporated into the DMD framework by enforcing a matrix manifold constraint determined by known physics. 
We have applied piDMD to five of the most fundamental physical principles: conservation laws,
self-adjointness, shift-invariance, locality, and causality.  
In several of these cases, we have presented new `exact' solutions of the corresponding optimization problems.
The examples presented demonstrate that piDMD can exhibit
superior performance compared to classical DMD methods that 
do not account for the physics of the system.

The framework developed herein offers exciting new opportunities and challenges for data-driven dynamical systems, numerical linear algebra, convex optimization, and statistics.
Fortunately, many research advances developed for classical DMD can also be leveraged to enhance piDMD. 
We conclude the paper with a brief discussion of the limitations of piDMD, and suggest promising future research directions.

\subsection{Limitations and extensions}

Each of the fours steps of the piDMD framework outlined in section \ref{Sec:piDMD} pose distinct challenges.
In some scenarios where the physics is poorly understood, determining suitable physical laws
to impose on the model (step 1, `modeling') can be challenging.
For problems with intricate geometries and multiple dimensions, interpreting the physical principle as a matrix manifold (step 2, `interpretation') can be the roadblock, as the manifold can become exceedingly complicated.
We have already commented on the challenges and opportunities posed by step 3 (`optimization'), and suggested algorithmic solutions \cite{Boumal2014}.
Finally, the typically large state dimension and number of samples can obfuscate step 4 when the matrix manifold does not easily admit the desired diagnostics. 
For example, if $\bA$ has a complicated banded structure, as in local piDMD, it may not be amenable to a fast eigendecomposition or SVD.

The solutions presented herein exhibit various levels of sensitivity to noise.
As is typical of DMD methods, most solutions are unbiased with respect to $\bY$, but can be quite sensitive to noise in $\bX$.
Some of the solutions we have presented are relatively insensitive (\S \ref{Sec:conservative}) or can be reformulated in the total least squares sense (\S \ref{Ap:circulantTLS}, \S \ref{Ap:localTLS}).
A full characterization of the sensitivity of piDMD should be performed on a case-by-case basis for different manifold constraints, and will be the subject of future investigations.

The Procrustes problem \eqref{Eq:procrustes} is not the only means of incorporating physical
principles into the DMD process.
Closely related to Procrustes problems are `nearest matrix' problems \cite{Higham1989}
where, given a matrix $\hat{\bA}$, we seek the
closest matrix on some matrix manifold $\cM$:
\begin{align}
	\argmin_{\bA \in \cM} \|\bA - \hat{\bA} \| .
	\label{Eq:nearest}
\end{align}
Equation \eqref{Eq:nearest} can be viewed as a special case of the Procrustes problem \eqref{Eq:procrustes} where $\bY = \hat{\bA}$ and $\bX = \bI$. As such, nearest matrix problems are generally easier
to solve than Procrustes problems.
The nearest matrix problem \eqref{Eq:nearest} can be used to project any DMD model onto the
nearest physically consistent model, or within optimization routines, such as proximal gradient
methods or constrained gradient descent, to constrain the solution to the feasible set.

As presented, piDMD strictly requires that the model $\bA$ lies on the chosen matrix manifold $\cM$.
In many applications, such as when the data are very noisy or the physical laws and/or constraints are only approximately understood, it may be more appropriate to merely `encourage' $\bA$ toward $\cM$.
In such cases, the piDMD regression \eqref{Eq:procrustes} becomes
\begin{align}
\argmin_{\bA} \|\bY - \bA \bX \|_F + \lambda \, R(\bA),
\end{align}
where the first term is the reconstruction loss, $R$ is a physics-informed regularizer, 
and $\lambda$ is a user-defined constant that tunes the relative importance
of the reconstruction and regularization.
For example, $R(\bA)$ could represent the distance between $\bA$ and $\cM$, which
can be computed by solving the nearest matrix problem \eqref{Eq:nearest}.
An example of such a problem is solved in appendix \ref{Ap:localGen}.

DMD models may be trained in discrete or continuous time.
In discrete time with a constant sampling rate, the snapshots in $\bY$
are $\by_k = \bx_{k+1}$, whereas in the continuous time we have $\by_k = \dot{\bx}_k$.
If the continuous time operator is $\bA_C$, then the discrete time operator is $\bA_D = \exp(\Delta t \bA_C)$.
In special cases, the matrices $\bA_C$ and $\bA_D$ may lie on the same manifold, such as the upper-triangular and circulant manifolds; however, in general, they do not.
For example, if $\bA_C$ is tridiagonal, then $\bA_D$ is not generally tridiagonal.
Accordingly, if only discrete-time measurements are available but the matrix manifold is imposed
on the continuous-time operator, then the optimization problem to solve is
\begin{align}
	\argmin_{\bA \in \cM} \|\bY - \exp (\Delta t \bA ) \bX \|_F ,
	\label{Eq:exponential}
\end{align}
which is a much more difficult optimization problem.

\section*{Acknowledgements} 
P.J.B. acknowledges insightful conversations with Suvrit Sra, Tasuku Soma and Andrew Horning.  
S.L.B. acknowledges valuable discussions with Jean-Christophe Loiseau. 
The authors acknowledge support from the Army Research Office (ARO W911NF-17-1-0306) and the National Science Foundation AI Institute in Dynamic Systems (Grant No. 2112085).

\appendix
\renewcommand{\thesection}{\Alph{section}}
\renewcommand{\thesubsection}{\Alph{section}.\arabic{subsection}}
\renewcommand{\thesubsubsection}{\Alph{section}.\arabic{subsection}.\arabic{subsubsection}}

\appendix
\section*{Appendices}
The following appendices provide further details of the optimization problems considered in this paper.

\section{Further details of shift-invariant systems}
Here we detail several extensions to the shift-invariant DMD problem introduced in section \ref{Sec:shiftInvariant}.
Figure \ref{Fig:shiftVariants} compares these alternative solutions for training data from the 2-D advection equation.
\subsection{Symmetric, skew-symmetric and unitary circulant matrices}
\label{Ap:circulantSymmetric}
Now consider the case where $\bA$ is symmetric
and circulant.
An application of \eqref{Eq:circulantDiag} reveals that a circulant matrix is symmetric if and only if its
eigenvalues $\{\hat{\ba}_i\}$ are real.
Thus, the minimization problem is
\begin{align}
	\argmin_{\Im[\hat{\ba}] = \boldsymbol{0}} \|\diag(\hat{\ba}) {\bXc} - {\bYc} \|_F^2.
\end{align}
Expanding the Frobenius norm row-wise and enforcing stationarity yields the optimal value as
\begin{align}
	\hat{a}_{j}  
	=\frac{\Re\left[\tilde{\bYc}_j\tilde{\bXc}_j^\ast \right] }
	{\left\| \tilde{\bXc}_j \right\|^2_2}.
\end{align}
The case of a skew-symmetric circulant matrix is similar 
except the eigenvalues are now imaginary; 
the solution is
\begin{align}
	\hat{a}_{j}  
	= \i \frac{\Im\left[\tilde{\bYc}_j\tilde{\bXc}_j^\ast \right] }
	{\left\|\tilde{\bXc}_j \right\|^2_2}.
\end{align}
Finally, for a unitary circulant matrix (when the eigenvalues lie on the unit circle), we have
\begin{align}
	\hat{a}_{j}  
	= \frac{\tilde{\bYc}_j\tilde{\bXc}_j^\ast }
	{\left\| \tilde{\bYc}_j \tilde{\bXc}_j^\ast \right\|_2}.
\end{align}
\begin{figure}[t]
	\centering
	\includegraphics[width=\linewidth]{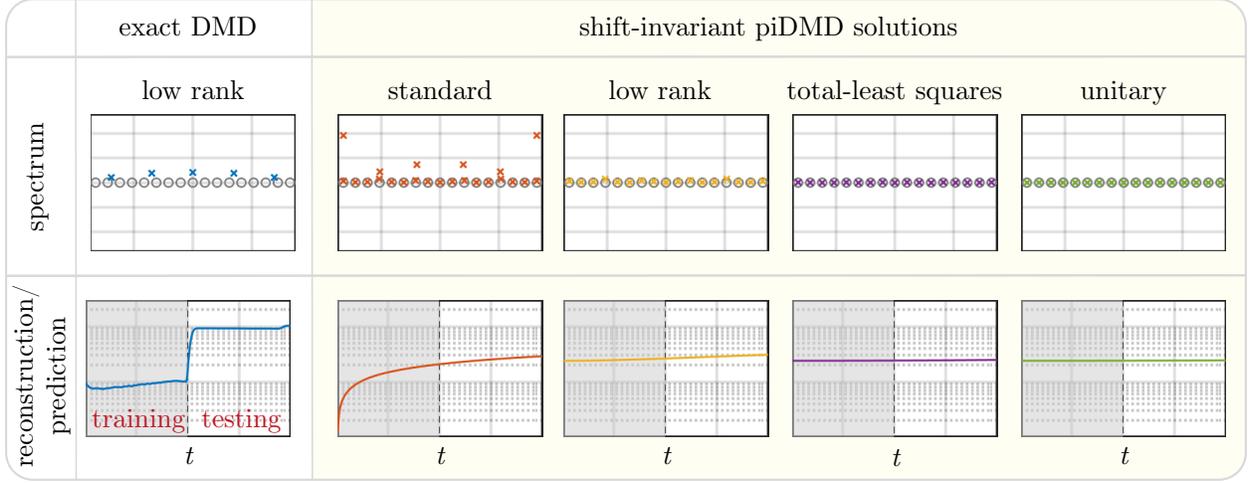}
	\caption{Comparison of exact DMD to shift-invariant piDMD and some variants.
	The model are trained on snapshots of a solution to the 2-D advection equation.
	The snapshots are contaminated by 2\% Gaussian noise.
The low-rank model has rank 45.}%
	\label{Fig:shiftVariants}
\end{figure}
\label{Ap:shiftInvariant}
\subsection{Best rank-\texorpdfstring{$r$}~~ circulant approximant}
\label{Ap:circulantLowRank}
Now consider the problem where $\bA$ is both circulant and rank $r\leq n$.
Since the rank of a matrix is the number of non-zero eigenvalues, 
the rank of a circulant matrix is the number of non-zero entries in $\hat{\ba}$.
Thus, only $r$ entries in $\hat{\ba}$ are permitted to be non-zero and
the minimization problem \eqref{Eq:circulant2} becomes
\begin{align}
	\argmin_{\|\hat{\ba}\|_0 = r} \|\diag(\hat{\ba}) \bXc - \bYc \|_F.
	\label{Eq:circulant3}
\end{align}
Equivalently, given the solution \eqref{Eq:circulantDiag}, we must select $r$ eigenvalues to maintain 
and set all other eigenvalues to zero.
To minimise the cost function, the $r$ chosen eigenvalues should contribute
to the largest reduction in the cost function.
The extent to which each eigenvector minimises the cost function is 
characterised by their residual:
\begin{align}
	{\rho}_j  = 
	\sqrt{	\|\tilde{\bYc}_j \|_2^2 -  \|\tilde{\bXc}_j - \hat{a}_j \tilde{\bXc}_j \|_2^2 } 
	= \frac{\Re\left[\tilde{\bYc}_j \tilde{\bXc}_j^\ast \right]}{\| \tilde{\bXc}_j \|_2}.
\end{align}
A large residual indicates a large reduction in the cost function and vice versa.
Therefore, the $r$ non-zero eigenvalues to maintain are the $r$ with the largest residuals.
\subsection{Learning shift-invariant operators with non-equispaced samples}
\label{Ap:nufft}
Here we extend the method of section \ref{Sec:shiftInvariant} to allow spatial samples that are not equally spaced.
We begin by formulating the Procrustes problem \eqref{Eq:procrustes} in continuous space.
Suppose that we have access to measurements of functions $u(\xi,t)$ and $v(\xi,t)$ for all $\xi \in [-1, 1]$
 but at discrete points $t \in \{t_j \ | \ j = 1, \dots, m \}$.
Then, the continuous-space piDMD regression \eqref{Eq:procrustes} is 
\begin{align}
	\argmin_{\mathcal{A}\in \cM} \sum_{j=1}^m \left\| v(\xi,t_j) - \mathcal{A} u(\xi, t_j) \right\|_2^2,
	\label{Eq:circulantCont}
\end{align}
where the continuous $2$-norm is
\begin{align}
	\|f(\xi) \|_2^2 = \int_{1}^{1} |f(\xi)|^2 \d \xi.
\end{align}
We now constrain $\mathcal{A}$ to be a shift-invariant operator on a periodic domain.
Such operators are diagonalized by exponential functions (\S \ref{Sec:shiftInvariant}) so we may write
\begin{align}
	\mathcal{A} f(\xi) = \sum_{k=-\infty}^\infty \hat{a}_k \e^{\i \pi k \xi}
	\hat{f}(k),
	\label{Eq:rep1}
\end{align}
where $\hat{f}(k)$ is the finite Fourier transform of $f$ at wavenumber $k$:
\begin{align}
	\hat{f}(k)
	= \int_{-1}^{1} f(\eta) \e^{- \i \pi k \eta}  \d \eta.
	\label{Eq:fourier}
\end{align}
Furthermore, if $u$ and $v$ are
sufficiently smooth in $\xi$ (e.g. they are H\"older continuous \cite{Zygmund1988})
then we may write
\begin{align}
    u(\xi, t)
    = \sum_{k=-\infty}^\infty \hat{u}(k, t) \e^{\i \pi k \xi},
    \qquad
    v(\xi, t)
    = \sum_{k=-\infty}^\infty \hat{v}(k, t) \e^{\i \pi k \xi}. \label{Eq:rep2}
\end{align}
Applying Parseval's theorem in concert with (\ref{Eq:rep1}) and  (\ref{Eq:rep2}) shows that the minimization problem \eqref{Eq:circulantCont} is equivalent to
\begin{align}
\argmin_{\hat{a}_k}	
\sum_{j=1}^m
\left\|
\sum_{k=-\infty}^\infty
\left(\hat{v}(k, t_j)- \hat{a}_k 
\hat{u}(k,t_j)
\right)
\e^{\i \pi k \xi}
\right\|_2^2
		  &= 
\argmin_{\hat{a}_k}	
\sum_{k=-\infty}^\infty
\sum_{j=1}^m
\left|\hat{v}(k, t_j)- \hat{a}_k
\hat{u}(k,t_j)
\right|^2.
\label{Eq:circulantCont2}
\end{align}
Furthermore, writing
$\hat{\bu} = \begin{bmatrix}\hat{u}(k, t_1) &\cdots &\hat{u}(k, t_m) \end{bmatrix}$ and
$\hat{\bv} = \begin{bmatrix}\hat{v}(k, t_1) &\cdots &\hat{v}(k, t_m) \end{bmatrix}$
transforms the right side of \eqref{Eq:circulantCont2} to
\begin{align*}
\argmin_{\hat{a}_k}	
\sum_{k=-\infty}^\infty
\left\|\hat{\bv}(k)- \hat{a}_k
\hat{\bu}(k)
\right\|_2^2,
\end{align*}
whose solution is
\begin{align}
    \hat{a}_k = \hat{\bv}(k) \hat{\bu}^\dagger(k) 
    = \left. \hat{\bv}(k) \hat{\bu}^\ast(k)\middle/ \|\hat{\bu}(k) \|_2^2 \right. .
    \label{Eq:circulantContSol}
\end{align}
Thus, \eqref{Eq:circulantContSol} solves the continuous-space shift-invariant piDMD regression \eqref{Eq:circulantCont}.
Note that the solution \eqref{Eq:circulantContSol}
depends only on the Fourier coefficients of $u$ and $v$ \eqref{Eq:rep2}.

Now, suppose that we only have
measurements of $u$ and $v$ at arbitrary discrete grid points $\xi \in \{\xi_l\}$.
We can then approximate the solution \eqref{Eq:circulantContSol}
by approximating the coefficients of the Fourier coefficients $\hat{\bv}(k)$ and $\hat{\bu}(k)$.
The samples at $\{\xi_l\}$ correspond to the quadrature rule for the Fourier coefficients \eqref{Eq:fourier}
\begin{align}
\hat{f}(k) \approx \sum_{l=1}^n w_l \ f(\xi_l) \e^{- \i \pi k \xi_l} \label{Eq:nufft}
\end{align}
where $w_k = (\xi_{k+1} - \xi_{k-1})/2 \ (\mathrm{mod}\ 1)$ are the quadrature nodes for the grid.
We can now use \eqref{Eq:nufft} to approximate $\hat{\bv}$ and $\hat{\bv}$ and thereby approximate the solution \eqref{Eq:circulantContSol}.
The sums \eqref{Eq:nufft} can be computed efficiently with
the non-uniform discrete Fourier transform (NUDFT--II, \cite{Dutt1993,Ruiz-Antolin2018}).

\subsection{Total least-squares circulant problem}
\label{Ap:circulantTLS}
A equivalent statement of the Procrustes problem \eqref{Eq:procrustes} is
\begin{align}
	\argmin_{\bA \in \cM}
	\| \bR\|_F \quad \textnormal{subject to  }
	\bA \bX = \bY + \bR.
\end{align}
As such, the formulation of \eqref{Eq:procrustes} assumes assumes that any errors are in $\bY$ only.
Alternatively, we can consider errors in both $\bX$ and $\bY$ and reformulate \eqref{Eq:procrustes} as
\begin{align}
	\argmin_{\bA \in \cM}
	\left\|\left[  \bR \;\;\; \bS \right]\right\|_F \quad \textnormal{subject to  }
	\bA \left(\bX + \bS \right) = \bY + \bR 
	\label{Eq:procrustesTLS}
\end{align}
This approach is usually more principled as there will be inevitable errors in measurements of $\bX$ and $\bY$.
Here we show that we can adapt the circulant solution of \S \ref{Sec:shiftInvariant} to account for noise in both sets of measurements.
Again, by the invariance of the Frobenius norm with respect to unitary transformations, \eqref{Eq:procrustesTLS} is equivalent to
\begin{align}
	\argmin_{\hat{\ba}}
	\left\|\left[  \bRc \; \; \; \bSc \right]\right\|_F \quad \textnormal{subject to  }
	\qquad
	\diag\left( \hat{\ba} \right) \left(\bXc + \bSc \right) = \bYc + \bRc
\end{align}
where $\bRc = \bFc^\ast \bR$ and $\bSc = \bFc^\ast \bS$.
Again, the rows decouple to produce $n$ smaller problems:
\begin{align}
	\argmin_{\hat{a}_i}
	\left\|\left[  \tilde{\bRc}_i \;\;\; \tilde{\bSc}_i \right]\right\|_F \quad \textnormal{subject to  } \quad
	\hat{a}_i \left(\tilde{\bXc}_i + \tilde{\bSc}_i \right) 
	= \tilde{\bYc}_i + \tilde{\bRc}_i.
	\label{Eq:procrustesTLSCirc}
\end{align}
for $1 \leq i \leq n$. These problems may now be solved in the total least squares sense \cite{Hemati2017tcfd,golub2012book}.

We will briefly characterize the scenario where the total least-squares circulant model is statistically optimal.
By optimal here, we mean that the model converges to the true model
in probability as the number of samples increases \cite{VanHuffel1991}.
In the unconstrained case, it is known that this is the case when the entries of $\bR$ and $\bS$ are
random variables with zero mean and
the covariance matrix of $\textnormal{vec}\left( [ \bR \;\;\; \bS ]\right)$
is a multiple of the identity matrix \cite{VanHuffel1991}.
Therefore, the solution of the total least squares problem \eqref{Eq:procrustesTLSCirc} 
will be statistically optimal if
\begin{align}
\textnormal{Var}\left({\bRc}_{i,j}\right) = \textnormal{Var}\left( {\bSc}_{i,j} \right) 
\qquad \textnormal{and}
\qquad \textnormal{Cov} \left( {\bRc}_{i,j},{\bSc}_{i,j} \right) = 0.
\end{align}
In words the circulant-TLS problem is statistically optimal if errors in $\bR$
and $\bS$ have equal variance for a given wavenumber and the errors between
$\textnormal{Cov} \left( {\bRc}_{i,j},{\bSc}_{i,j} \right)$ are independent.
These conditions are met if $\bR$ and $\bS$ are independent random variables of equal variance $\sigma$, in which case we have
$\textnormal{Var}\left({\bRc}_{i,j}\right) = \textnormal{Var}\left( {\bSc}_{i,j} \right) = \sigma^2$.
Therefore, if the unconstrained total least squares model is statistically optimal then so is the circulant piDMD model.

\subsection{Toeplitz and Hankel matrices}
\label{Ap:toeplitz}
Here we adapt the optimization of section \S \ref{Sec:shiftInvariant} to  Toeplitz matrices, which take the form
\begin{align}
\bA=
\begin{bmatrix}
	a_{0}&a_{-1}&\cdots &a_{-(n-1)}\\
	a_{1}&a_0&\ddots&\vdots \\
	\vdots &\ddots&\ddots&a_{-1}\\
	a_{n-1}&\cdots &a_{1}&a_{0}
\end{bmatrix}.
\end{align}
Similarly to the circulant case, $\bA_{i,j}= a_{i-j}$ but now the subscripts not are considered modulo $n$.
The Procrustes problem for a Toeplitz matrix was considered by \cite{Yang2013}.
We produce an alternative solution based on the fast Fourier transform, analogous to the solution
for the circulant Procrustes problem (\S \ref{Sec:shiftInvariant}).
A Toeplitz matrix can be embedded in a circulant matrix as
\begin{align}
	{\renewcommand\arraystretch{1.3}
	\bC = \mleft[
	\begin{array}{c|c}
		\bA & \tilde{\bA} \\\hline
		\tilde{\bA} & \bA
\end{array}
\mright]}
\qquad \textrm{ where } \qquad
\tilde{\bA}=
\begin{bmatrix}
{0}&a_{n-1}& \cdots &a_{1}\\
a_{-(n-1)}&0&\ddots &\vdots \\
\vdots &\ddots & \ddots & a_{n-1}\\
a_{-1} & \cdots & a_{-(n-1)} & 0
\end{bmatrix}
\end{align}
Similarly to \S \ref{Sec:shiftInvariant}, the circulant matrix $\bC$ is diagonalized by the DFT  matrix $\bFc$
(which is now $2n\times 2n$).
Thus, we may write $\bC = \bFc^\ast \diag(\hat{\ba}) \bFc$
and the minimization \eqref{Eq:procrustes} becomes
\begin{align}
	\argmin_{\hat{\ba}} \left\|
	\hat{\bI}^\ast\diag(\hat{\ba}) {\bXc} - \bY
	\right\|_F^2
	\label{Eq:costToeplitz}
\end{align}
where $\hat{\bI} =\bFc \begin{bmatrix} \bI \\ \boldsymbol{0} \end{bmatrix}$
and ${\bXc} = \bFc \begin{bmatrix} \bX \\ \boldsymbol{0} \end{bmatrix}$.
It can be shown that
the optimal coefficients $\hat{\ba}$ solve the linear system
\begin{align}
\bH \hat{\ba} = \bd
\qquad \textrm{ where }
\qquad 
\bH = \left(\hat{\bI} \hat{\bI}^\ast \odot \overline{{\bXc} {\bXc}^\ast}\right),\quad
\bd = \textrm{diag} \left( \hat{\bI} \bY {\bXc}^\ast \right)
\label{Eq:toeplitzSol}
\end{align}
and $\odot$ represents element-wise multiplication (i.e. the Hadamard product).
Note that a similar optimization arises in the computation of the amplitudes of DMD modes \cite{jovanovic2014pof}.
While mathematically correct, numerical implementations of \eqref{Eq:toeplitzSol}
can be difficult. Specifically, issues arise because the Gramian matrix ${\bXc} {\bXc}^\ast$ will typically be large and ill-conditioned. 
Thus, we expect that one could find more efficient algorithms to calculate $\bA$, 
possibly by adapting methods for fast solutions of Toeplitz systems
\cite{Levinson1946,Brent1980,Stewart2004}.

This approach also applies to Hankel matrices, which take the form
\begin{align}
\bA=\begin{bmatrix}
a_{0}&a_{1}&\ldots &a_{n-1}\\
a_{1}&a_{2}&\iddots& \vdots \\
\vdots & \iddots & \iddots &a_{2n-3}\\
a_{n-1}&\ldots &a_{2n-3}&a_{2n-2}
\end{bmatrix}
\end{align}
i.e. each ascending skew-diagonal is constant: $\bA_{i,j} = \ba_{i+j-2}.$
Hankel matrices are upside-down Toeplitz matrices: if $\bA$ is Hankel, then $\bJ \bA$ is Toeplitz, where $\bJ$ is ths upside-down identity matrix.
Thus, the Procrustes problem for a Hankel matrix is equivalent to the for a Toeplitz matrix
with a suitable flipping of the data.

\section{Further details of self-adjoint systems}
\label{Ap:selfAdjoint}

\subsection{Skew-symmetric systems}
\label{Ap:skewSymmetric}
If $\bA$ is constrained to be skew-symmetric then \eqref{Eq:sym1} and 
still holds and $\bC = \bU_X^\ast \bY \bV_X$ but now $\bL$ is now defined by
\begin{align}
	\bL_{i,j} = -\overline{\bL}_{j,i} = \begin{cases}
		\dfrac{-\sigma_i \overline{\bC}_{j,i} + \sigma_j \bC_{i,j}}{\sigma_i^2 + \sigma_j^2}
		& \textnormal{if } \sigma_i^2 + \sigma_j^2 \neq 0,\\
		\qquad \qquad 0 & \textnormal{otherwise}.
\end{cases}
\label{Eq:symS}
\end{align}



\subsection{Sensitivity of symmetric piDMD}
In light of the results of section \ref{Sec:selfAdjoint}, we now compare the sensitivities of the symmetric piDMD model \eqref{Eq:sym1} and exact DMD model \eqref{Eq:exactDMD}.
In particular, we will show that the elements of the piDMD model have lower variance than those of the corresponding exact DMD model.
Higham \cite{Higham1988} also investigated the sensitivity of the symmetric Procrustes problem, but from the perspective of numerical stability; here we take a statistical perspective.
In this section only, we denote the exact DMD solution \eqref{Eq:exactDMD} as $\bA^e$ and the symmetric piDMD solution \eqref{Eq:sym1} as $\bA^s$.

We assume that measurements of $\bY$ are contaminated with I. I. D. Gaussian noise.
For now, we neglect noise in $\bX$, which could be incorporated via a total least square approach \cite{Hemati2017tcfd,Dawson2016},
By linearity, it is sufficient to consider $\bY_{i,j} \sim \mathcal{N}(0, 1)$.
Again, since exact DMD and symmetric piDMD are linear models,
they are unbiased estimators:
\begin{align*}
    \E \left(\bA^{s}_{i,j} \right)
    =
    \E \left(\bA^{e}_{i,j} \right)
    = 0.
\end{align*}
We now compare the variance of the models.
Extensive algebra produces explicit expressions for the variances of $\bA^{e}_{i,j}$ and $\bA^{s}_{i,j}$ as
\begin{gather}
	\Var \left( {\bA^e}_{i,j} \right)
	= \sum_{l = 1}^{r} 
	\frac{\bU_{j,l}^2}{\sigma_l^2},\label{Eq:exactVar}
 \\[2ex]
	\Var(\bA^s_{i,j}) = \frac{1}{2} \sum^{r}_{k=1} \sum^{r}_{l=1}  \frac{\left(\bU_{i,k} \bU_{j,l} 
	+ \bU_{i,l} \bU_{j,k}\right)^2}{\sigma_k^2 + \sigma_l^2},
	\label{Eq:symmetricVar}
\end{gather}
where $\bU$ is the matrix of left singular vectors of $\bX$, $\{\sigma_l\}$ are the associated singular values, and $\textrm{rank}(\bX) = r$.
Figure \ref{Fig:symmetricNoise} compares these theoretical values to empirical values for random data.
These results are in agreement with figure \ref{Fig:Schrodinger} and indicate that the variance of the piDMD model is smaller than that of the exact DMD model.
\begin{figure}
    \centering
    \includegraphics[width = \linewidth]{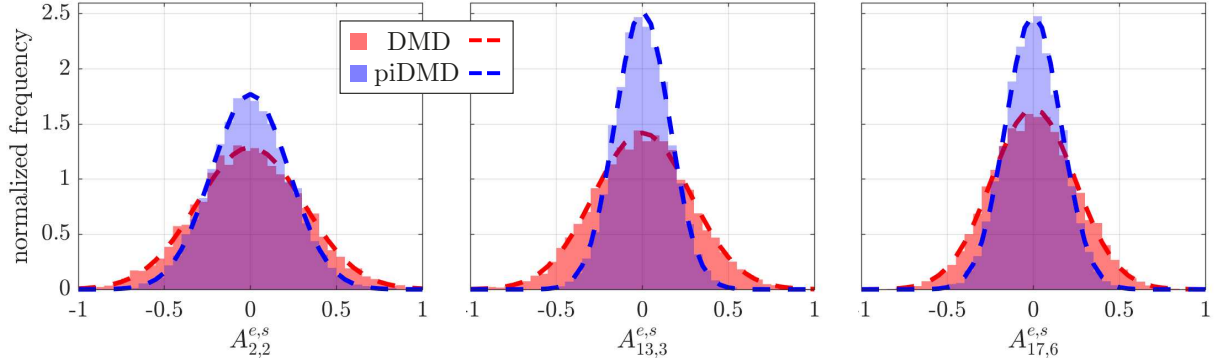}
    \caption{Distribution of errors in the entries of the learned models ($A^{e,s}_{2,2}$, $A^{e,s}_{13,3}$ and $A^{e,s}_{17,6}$) for exact DMD and self-adjoint piDMD. The block colors represent the frequency and the dashed lines are the theoretical curves (\ref{Eq:exactVar}, \ref{Eq:symmetricVar}).}
    \label{Fig:symmetricNoise}
\end{figure}
We formalize this result in the following lemma and proof.

\begin{lemma} \label{NoiseLemma}
If the measurements of $\bY$ are contaminated with I. I. D. Gaussian noise, the variance of each element of the symmetric piDMD model \eqref{Eq:sym1} is bounded above by the variance of the corresponding element of the exact DMD model \eqref{Eq:exactDMD}:
\begin{align}
	\Var(\bA_{i,j}^s) \leq \Var({\bA}_{i,j}^e).
\end{align}
\end{lemma}
\begin{proof}
Sedrakyan's inequality \cite{Sedrakyan1997} produces
\begin{align*}
\frac{\left(\bU_{i,k} \bU_{j,l} 
+ \bU_{i,l} \bU_{j,k}\right)^2}{\sigma_k^2 + \sigma_l^2}
& \leq \frac{\bU_{i,k}^2\bU_{j,l}^2}{\sigma_l^2}
+
\frac{\bU_{i,l}^2\bU_{j,k}^2}{\sigma_k^2}.
\end{align*}
%
Therefore,
\begin{align}
	\Var(\bA^s_{i,j}) \leq
	\frac{1}{2} \sum^{r}_{k=1} \sum^{r}_{l=1}  
	\left(
	\frac{\bU_{i,k}^2\bU_{j,l}^2}{\sigma_l^2}
+
\frac{\bU_{i,l}^2\bU_{j,k}^2}{\sigma_k^2}
	\right)
	=\sum^{r}_{k=1} \sum^{r}_{l=1}  
	\frac{\bU_{i,k}^2\bU_{j,l}^2}{\sigma_l^2}
    =\left( \sum^{r}_{k=1} \bU_{i,k}^2 \right) 
    \Var(\bA^e_{i,j}).
\end{align}
The result now follows since the rows of $\bU$ have unit norm so $\sum_{k=1}^r \bU_{i,k}^2 \leq 1$.
\end{proof}

This result indicates that the symmetric piDMD model is less sensitive to noise than the exact DMD model when the noise is only in $\bY$. Future work should investigate the case when both $\bX$ and $\bY$ are contaminated with noise, and clarify the effect of noise on the learned spectrum.

\section{Further details of local systems}
\label{Ap:local}

\subsection{More general local structures}
\label{Ap:localVar}
The local solution of section \ref{Sec:local} naturally extends
to more general local structures.
For example, it may be necessary to consider longer range interactions between states and allow, say, a state to depend on other states three grid points away.
Additionally, there may be a variable local dependence where different states 
can depend on different numbers of elements.
For example, a state near boundary (e.g. $x_1$ or $x_n$) may need to depend on more than one
other state in order to properly resolve the boundary behaviour.
To make this precise, suppose that state $i$ depends on $l(i)$ states to the right
and $u(i)$ states to the left above the leading diagonal
(assuming that the data is organised from left to right).
Then, the solution for the optimal matrix entries is
\begin{align}
	\begin{bmatrix}
		A_{i-l(i),i} & 
		\cdots &
		A_{i,i} &
		\cdots &
		A_{i+u(i),i}
	\end{bmatrix}
	= \tilde{\by}_i 
	\begin{bmatrix}
		\tilde{\bx}_{i-l(i)}\\
		\vdots\\
		\tilde{\bx}_{i}\\
		\vdots\\
		\tilde{\bx}_{i+u(i)}
	\end{bmatrix}^\dagger
	\qquad\qquad
	\textnormal{for }
	1\leq i\leq n
\end{align}
with the convention $A_{i,j} = 0$ if the entry does not exist and $\tilde{\bx}_i = \boldsymbol{0}$ for $i<1$ or $i>n$.
We recover the tridiagonal solution \eqref{Eq:localSol} when $l(i) = u(i) = 1$.

\subsection{Local and periodic}
\label{Ap:localPer}
The tridiagonal solution in section \ref{Sec:local} effectively
assumed that the system had two endpoints that were 
disconnected from one another.
However, the spatial domain could be periodic, in which case 
the first state can affect the $n$th
state and vice versa.
If we allow this structure then the matrix $\bA$ takes the form
\begin{align}
	\bA = 
	 \begin{bmatrix}
		 \beta_1 & \gamma_1 & & & \alpha_1 \\
\alpha_2 & \beta_2& \gamma_2 \\
& \alpha_3 & \ddots & \ddots \\
& & \ddots & \ddots & \gamma_{n-1} \\
\gamma_n & & & \alpha_{n} & \beta_n
\end{bmatrix}.
\end{align}
The solution for the optimal constants is identical to 
those stated in \eqref{Eq:localSol} except now the subscripts are evaluated modulo $n$.

\subsection{Symmetric tridiagonal}
\label{Ap:localSym}
Many spatially local systems are also self-adjoint.
In such cases, we may enforce that $\bA$ is both tri-diagonal and symmetric:
\begin{align}
	\bA = 
	 \begin{bmatrix}
\alpha_1 & \beta_1 \\
\beta_1 & \alpha_2& \beta_2 \\
& \beta_2 & \ddots & \ddots \\
& & \ddots & \ddots & b_{n-1} \\
& & & b_{n-1} & \alpha_n
\end{bmatrix}
\end{align}.
The cost function \eqref{Eq:procrustes} is then expressible as
\begin{align}
	\| \bY - \bA \bX  \|_F^2 &= 
 \|\tilde{\by}_1 - \alpha_1 \tilde{\bx}_1 - \beta_1 \tilde{\bx}_2  \|_2^2 
+ \| \tilde{\by}_n - b_{n-1} \tilde{\bx}_{n-1} - \alpha_n \tilde{\bx}_n \|_2^2 \\
&+ \sum_{i = 2}^{n-1} \|\tilde{\by}_i -b_{i-1} \tilde{\bx}_{i-1} 
- \alpha_i \tilde{\bx}_i - b_{i} \tilde{\bx}_{i+1} \|_2^2.
\end{align}
Enforcing stationarity in $\bc = \begin{bmatrix} \alpha_1 & \cdots & \alpha_n & \beta_1 & \cdots & b_{n-1}\end{bmatrix}^T$ 
produces the linear system
\begin{align}
	\bT \bc = \bd \label{Eq:symtrisys}
\end{align}
where $\bT$ is the symmetric block-tridiagonal matrix
\begin{align}
	{\renewcommand\arraystretch{1.3}
	\bT = \mleft[
	\begin{array}{c|c}
		\bT_1 & \bT_2 \\\hline
		\bT_2^\ast & \bT_3
\end{array}
\mright]}.
\end{align}
where the blocks are
\begin{gather}
\bT_1 = 
\mathrm{diag}\left( \|\tilde{\bx}_1\|_2^2\, , \, \cdots \,,\, \|\tilde{\bx}_n\|_2^2\right),\\[2ex]
\bT_2 = 
 \begin{bmatrix}
\tilde{\bx}_2^\ast \tilde{\bx}_1 &  \\
\tilde{\bx}^\ast_2 \tilde{\bx}_1 & \tilde{\bx}_3^\ast \tilde{\bx}_2 &  \\
  & \tilde{\bx}_4^\ast \tilde{\bx}_3 & \ddots & \\
& & \ddots & \ddots & \\
& & & \tilde{\bx}_{n-1}^\ast \tilde{\bx}_{n-2} & \tilde{\bx}_{n}^\ast \tilde{\bx}_{n-1}\\
& & & & \tilde{\bx}_n^\ast \tilde{\bx}_{n-1}
\end{bmatrix},\\[2ex]
\bT_3 =  
 \begin{bmatrix}
 \tilde{\bx}_2^\ast \tilde{\bx}_2 +
 \tilde{\bx}_1^\ast \tilde{\bx}_1& \tilde{\bx}_3^\ast \tilde{\bx}_1  \\
 \tilde{\bx}_3^\ast \tilde{\bx}_1 
& \tilde{\bx}_3^\ast \tilde{\bx}_3 + \tilde{\bx}_2^\ast \tilde{\bx}_2 
& \tilde{\bx}_4^\ast \tilde{\bx}_2  \\
& \tilde{\bx}_4^\ast \tilde{\bx}_2 & \ddots & \ddots & & \\
& & \ddots & \ddots &  \tilde{\bx}_{n}^\ast \tilde{\bx}_{n-1} \\
& & & \tilde{\bx}_{n}^\ast \tilde{\bx}_{n-1} & \tilde{\bx}_{n}^\ast \tilde{\bx}_{n-1},
\end{bmatrix}
\end{gather}
and
\begin{align*}
	\bd = 
	\begin{bmatrix}
		\tilde{\bx}_1^\ast \tilde{\by}_1 &
		\cdots &
		\tilde{\bx}_n^\ast \tilde{\by}_n &&
		\tilde{\bx}_1^\ast \tilde{\by}_2
		+ \tilde{\bx}_2^\ast \tilde{\by}_1 &
		\cdots &
		\tilde{\bx}_{n-1}^\ast \tilde{\by}_n
		+ \tilde{\bx}_n^\ast \tilde{\by}_{n-1}
	\end{bmatrix}^\ast.
\end{align*}
Note that $\bT_1 \in \mathbb{R}^{n\times n}$, 
$\bT_2 \in \mathbb{R}^{n\times(n-1)}$,
and $\bT_3 \in \mathbb{R}^{(n-1)\times(n-1)}$.
The system \eqref{Eq:symtrisys} can be solved very efficiently by exploiting the sparsity and symmetry of $\bT$.

\subsection{Local total least squares}
\label{Ap:localTLS}
This estimator is unbiased if there is only noise in $\bY$.
If $\bX$ and $\bY$ are both contaminated by noise then \eqref{Eq:tridiagonalMin} can be suitably adapted
and solved via total least squares \cite{VanHuffel1991}, provided that the noise is normally distributed 
and has the same (diagonal) covariance matrix for $\bX$ and $\bY$.
\subsection{Weaker spatial locality}
\label{Ap:localGen}
A weaker version of spatial locality can also be enforced by penalising
entries via
\begin{align}
	\argmin_{\bA} \|\bY - \bA \bX \|_F + \lambda\|\bH \odot \bA \|_F 
	\label{Eq:localWeak}
\end{align}
where $\bH$ is a matrix that penalises the solution
for grid points that are spatially far from one another
and $\lambda$ is a regularization parameter.
For example,  a suitable choice
for $\bH$ could be the Gaussian kernel matrix
\begin{align}
\bH_{i,j} = \exp\left( \|\xi_i - \xi_j \|_2/(2 \sigma^2 )\right)
\end{align}
for some parameter $\sigma$, where $\{\xi_i\}$ are the known grid points.
For general $\bH$, the solution to \eqref{Eq:localWeak} is 
\begin{align*}
    \tilde{\bA}_i = 
    \tilde{\by}_i \bX^\ast 
    \left(
    \bX \bX^\ast + \textnormal{diag}(\tilde{\bh}_i )
    \right)^{-1}
    \label{Eq:weakSol}
\end{align*}
for $1\leq i \leq n$ where $\tilde{\bh}_i$ is the $i$-th row of $\bH$.
Unfortunately, since the regularizer penalises
every entry of $\bA$, the implementation of \eqref{Eq:weakSol} costs $\mathcal{O}(n^4)$ operations 
to compute the full model $\bA$.
Thus, we recommend pursuing inexact but efficient solutions, possibly using descent methods \cite{BoydBook}.

\section{Further details for causal systems}
\subsection{Updating equations for upper triangular system}
\label{Ap:triangularUpdate}
Here we present an efficient implementation
of the solution {\eqref{Eq:triangularSol}}.
Consecutive blocks of $\bX$ (e.g. $\bX_{j-1:n,:}$ and $\bX_{j:n,:}$) are related by the rank-1 update
\begin{align}
	\bX_{j-1:n,:} = \begin{bmatrix}
	\boldsymbol{0}\\
	\bX_{j:n,:}
	\end{bmatrix}
	+ 
	\begin{bmatrix}
		1\\  \boldsymbol{0}
	\end{bmatrix}
	\tilde{\bx}_{j-1}.
\end{align}
This observation enables us to efficiently calculate the pseudoinverse $\tilde{\bX}_{j-1:n,:}^\dagger$ from $\tilde{\bX}_{j:n,:}^\dagger$.
The following formulas come from \cite{Meyer1973} (after some simplification).
Let $\hat{\bx}_{j-1} = \tilde{\bx}_{j-1}\left( \bI - \bX_{j:n,:}^\dagger \bX_{j:n,:} \right)$
be the orthogonal projection of $\tilde{\bx}_j$ onto the complement of the row space of $\bX_{j:n,:}$.
If $\tilde{\bx}_{j-1}$ lies in the row space of $\bX_{j:n,:}$ (i.e.~$\|\hat{\bx}_{j-1} \|_2\neq0$)
then
\begin{align}
	\bX_{j-1:n,:}^\dagger = \begin{bmatrix}
		\hat{\bx}_{j-1}^\dagger && \left( \bI - \hat{\bx}_{j-1}^\dagger \tilde{\bx}_{j-1} \right)\bX_{j:n,:}^\dagger
	\end{bmatrix}
	\label{Eq:update1}
\end{align}
otherwise,
\begin{align}
	\bX_{j-1:n,:}^\dagger = \begin{bmatrix}
		\boldsymbol{0} & \bX_{j:n,:}^\dagger
	\end{bmatrix}
	\left( \bI - \bv_j^\dagger \bv_j \right)
	\label{Eq:update2}
\end{align}
where $\bv_j = \begin{bmatrix}1 & -\tilde{\bx}_{j-1} \bX_{j:n,:}^\dagger \end{bmatrix}$.
Note that \eqref{Eq:update1} and \eqref{Eq:update2} can be formed using only matrix-vector products.
It turns out that the above formulation is unstable for realistic data 
measurements (with high condition number).
Below, we present an alternative algorithm that is  as fast but more stable.

\subsection{Alternative (more) stable solution}
\label{Ap:triangularStable}
We now derive an alternative solution to the upper-triangular piDMD problem of section \ref{Ap:triangular}.
In our experience, this version is more stable than that of the previous section.
We use the economy RQ decomposition of $\bX$ and write $\bX = \bR \bQ$
where $\bR \in \mathbb{C}^{n\times n}$ 
is upper triangular, and $\bQ \in \mathbb{C}^{n\times m}$ satisfies $\bQ \bQ^\ast = \bI$.
Note that
\begin{align}
	\|\bY - \bA \bX \|_F^2 
	&=\|\bY\|_F^2 - \|\bY\bQ^\ast\|_F^2
       +  \| \bY \bQ^\ast - \bA \bR \|_F^2. \label{Eq:ut3}
\end{align}
The first two terms in \eqref{Eq:ut3} are independent of $\bA$ and, by the submultiplicativity of 
the Frobenius norm\footnote{In particular, $\|\bC \bD \|_F \leq \| \bC\|_F \|\bD\|_2$ for any matrices $\bC$ and $\bD$.
Note the $2$-norm in the second term on the right.}, have a non-negative sum.
Thus, the upper-triangular Procrustes problem may be phrased as
\begin{align}
	\argmin_{\bA \in \mathcal{M}} \| \bY - \bA \bX \|_F =
	\argmin_{\bA \in \mathcal{M}} \| \bY \bQ^\ast - \bA \bR \|_F.
	\label{Eq:ut}
\end{align}
Now, since the product of upper triangular matrices is upper triangular, $\bA \bR$ is also  upper triangular.
Thus, the lower triangular component of $\bY \bQ^\ast$ is irrelevant to the minimization problem and we may write
the problem as
\begin{align}
	\argmin_{\bA \in \mathcal{M}} \| \texttt{triu}\left( \bY \bQ^\ast \right) - \bA \bR \|_F,
	\label{Eq:ut2c}
\end{align}
where, in \textsc{Matlab} notation, $\texttt{triu}(\cdot)$ extracts the upper triangular portion of its argument.
Since the rows of $\bA$ are independent, the Frobenius norm may be expanded row-wise and minimized separately.
Thus, the rows of $\bA$ solve the $n$ smaller minimization problems
\begin{align}
	\argmin_{{\bA}_{i,:}} \| 
\left( \bY \bQ^\ast \right)_{i,i:n} - \bA_{i,i:n} {\bR}_{i}\|_F
\qquad \qquad \textnormal{for } 1\leq i \leq n,
\label{Eq:ut1a}
\end{align}
where $\bR_i = \bR_{i:n,i:n} \in \mathbb{R}^{(n-i+1) \times n-i+1}$ is the $i$-th bottom right block of $ \bR$
and \mbox{$\bA_{i,i:n} \in \mathbb{R}^{n-i+1}$} is the $i$-th row of $\bA$ with the $i-1$ leading zeros removed.
Each minimization admits a unique solution if each $\bR_{i}$ is full rank, which occurs if the pivots of $\bR$ are non zero, i.e. $\bX$ is full rank.
If the data is rank deficient, as is often the case in applications, there are infinitely many solutions.
In any case, we seek the smallest solution (in the Frobenius norm sense), which is given by
\begin{align}
	\bA_{i,i:n} = \left( \bY \bQ^\ast \right)_{i,i:n} {\bR}_{i}^\dagger.
	\label{Eq:ut2}
\end{align}
Calculating $\bA$ via \eqref{Eq:ut2} naively requires computing each pseudoinverse $n$ times for a total of
$\mathcal{O}(n^4)$ operations.
As in \S \ref{Ap:triangularUpdate}, this complexity can be lowered with an updating procedure:
following \cite{Meyer1973}, we can efficiently compute ${\bR}_i^\dagger$ using ${\bR}_{i+1}^\dagger$.
Thus, we iterate through \eqref{Eq:ut2} backwards, starting at $i=n$.
when the (scalar) pseudoinverse is given by ${\bR}_{n,n}^\dagger = 1/{\bR}_{n,n}$.
Proceeding inductively, we note that the next matrix block can be expressed as a rank-one update to the
previous block:
\begin{align}
	{\bR}_{i} = 
	 \begin{bmatrix}
		0 & \boldsymbol{0} \\
		\boldsymbol{0} & {\bR}_{i+1}
	\end{bmatrix}
	 + \begin{bmatrix}
		 1 \\ \boldsymbol{0}
	 \end{bmatrix}
	 {\bR}_{i,i:n}.
\end{align}
The updating formula takes one of two forms depending on whether the pivot ${\bR}_{i,i}$ is 
smaller than some tolerance $\epsilon_{\textrm{tol}}$ close to machine precision:
\paragraph{Case 1: \textnormal{ $|{\bR}_{i,i} |> \epsilon_{\textrm{tol}}$}}
\begin{align}
	{\bR}_{i}^\dagger = 
	\renewcommand\arraystretch{2}
	 \begin{bmatrix}
		 1/{\bR}_{i,i} & {\bR}_{i,i:n}^\ast {\bR}_{i+1}^\dagger/{\bR}_{i,i}  \\
		 \boldsymbol{0} & {\bR}_{i+1}^\dagger
	\end{bmatrix}.
	\label{Eq:utCase1}
\end{align}
\paragraph{Case 2:\textnormal{ $|{\bR}_{i,i} |< \epsilon_{\textrm{tol}}$}}
\begin{align}
	{\bR}_{i}^\dagger = 
	\renewcommand\arraystretch{1}
	 \begin{bmatrix}
		 0 &  \boldsymbol{0} \\
		 \boldsymbol{0} & {\bR}_{i+1}^\dagger
	\end{bmatrix}
	\left( \bI
	-\frac{\boldsymbol{\alpha}_i
	\boldsymbol{\alpha}_i^\ast}{\|\boldsymbol{\alpha}_i \|_2^2}
	\right)
	\label{Eq:utCase2}
\end{align}
where $\boldsymbol{\alpha}_i =  ( {\bR}_{i,i:n} {\bR}_{i+1}^\dagger )^\ast -\begin{bmatrix} 1 & \boldsymbol{0} \end{bmatrix}^T$.
While mathematically correct, the accuracy of \eqref{Eq:utCase1} and \eqref{Eq:utCase2}
can deteriorate over many iterations.
Thus, in applications it may be sensible to reform the pseudoinverse ${\bR}_i^\dagger$ from scratch every so often.

The eigenvalues of $\bA$ can also be computed without forming $\bA$ explicitly.
Note that if ${\bR_{i,i}} \neq 0$ then the upper-triangular structure of ${\bR}_i$ implies that 
$\bA_{i,i} = \left( \bY \bQ^\ast \right)_{i,i}/\bR_{i,i}$.
Further, the eigenvalues of a triangular matrix are its diagonal entries, so $\{\bA_{i,i}\}$ are the eigenvalues of $\bA$ (for $\bR_{i,i} \neq 0$)
and can be computed without forming $\bA$ fully.


 \begin{spacing}{.9}
 \small{
 \setlength{\bibsep}{3.5pt}
\bibliographystyle{ieeetr}
 \bibliography{library}
 }
 \end{spacing}
\end{document}